\documentclass[11pt]{amsart}
\usepackage{amssymb,amsmath,amscd,xy,graphicx,textcomp}

\usepackage{tikz}\usetikzlibrary{matrix}
\usepackage{tikz-cd}
\usepackage{xcolor}

\usepackage{tensor}

\usepackage{float}
                                                       
\usepackage{ hyperref}

\newtheorem{theorem}{Theorem}[section]
\newtheorem{lemma}[theorem]{Lemma}
\newtheorem{corollary}[theorem]{Corollary}
\newtheorem{definition}[theorem]{Definition}

\newtheorem{remark}[theorem]{\it Remark}
\newtheorem{example}[theorem]{Example}
\newtheorem{proposition}[theorem]{Proposition}

\xyoption{arrow}

\xyoption{matrix}

\def\Gr{\mathrm{Gr}}
\def\GL{\mathrm{GL}}

\def\SL{\mathrm{SL}}

\def\SU{\mathrm{SU}}
\def\U{\mathrm{U}}

\def\t{\mathrm}

\def\C{\mathbb{C}}
\def\R{\mathbb{R}}
\def\Z{\mathbb{Z}}

\def\Z{\mathbb{Z}}

\def\tree{\mathcal{T}}

\def\q{/\!\!/}
\def\qb{\! \big/\hspace{-.15cm}\big/}
\def\ql{\backslash \!\! \backslash}
\def\qlb{\hspace{.05cm} \big\backslash \hspace{-.3cm} \big\backslash}

\def\hc{\operatorname{hc}}
\def\sc{\operatorname{sc}}

\def\Proj{\mathrm{Proj} \,}

\newcommand{\Spec}{\operatorname{Spec}}
\newcommand{\End}{\operatorname{End}}
\newcommand{\As}{\operatorname{As}}

\newcommand{\Genh}{G_{\operatorname{enh}}}
\newcommand{\Kenh}{K_{\operatorname{enh}}}
\newcommand{\Aenh}{A_{G_{\operatorname{enh}}}}
\newcommand{\A}{\mathbb{A}}

\newcommand{\abs}{\operatorname{abs}}

\newcommand{\an}{\operatorname{an}}
\newcommand{\alg}{\operatorname{alg}}

\newcommand{\gr}{\operatorname{gr}}

\renewcommand{\Re}{\mathfrak{R}\!\operatorname{e}}
\renewcommand{\Im}{\mathfrak{I}\!\operatorname{m}}

\newcommand{\Hom}{\operatorname{Hom}}
\newcommand{\Tr}{\operatorname{Tr}}

\newcommand{\Sym}{\operatorname{Sym}}

\newcommand{\Id}{\operatorname{Id}}

\title{Contraction of Hamiltonian $K$-spaces}
\author{Joachim Hilgert} 
\address[Joachim Hilgert]{Institut f\"ur Mathematik, 
 Universit\"at Paderborn,
33095 Paderborn, Germany}
\email{hilgert@math.uni-paderborn.de}
\author{Christopher Manon}
\address[Christopher Manon]{Department of Mathematics,
George Mason University,
Fairfax, VA 22030, USA}
\email{cmanon@gmu.edu}
\author{Johan Martens}
\address[Johan Martens]{School of Mathematics and Maxwell Institute, The University of Edinburgh, Peter Guthrie Tait Road, Edinburgh EH9 3FD, United Kingdom}
\email{johan.martens@ed.ac.uk}

\date{\today}
\thanks{}

\begin{document}

\begin{abstract}
In the spirit of recent work of Harada-Kaveh and Nishinou-Nohara-Ueda, we study the symplectic
geometry of Popov's horospherical degenerations of complex algebraic varieties with the action of a complex linearly reductive group. We formulate an intrinsic symplectic contraction of a Hamiltonian space, which is a surjective, continuous map onto a new Hamiltonian space that is a symplectomorphism on an explicitly defined dense open subspace.  This map is given by a precise formula, using techniques from the theory of symplectic reduction and symplectic implosion. 
We then show, using the Vinberg monoid, that the gradient-Hamiltonian flow for a horospherical degeneration of an algebraic variety gives rise to this contraction from a general fiber to the special fiber. We apply this construction to branching problems in representation theory, and finally we show how the Gel'fand-Tsetlin integrable system can be understood to arise this way.
\end{abstract}

\maketitle

\tableofcontents

\section{Introduction}

Flat degeneration is a powerful tool for studying algebraic varieties.  Heuristically, this method finds success because many constructions of flat degenerations, e.g. Gr\"obner bases and SAGBI (subalgebra analog to Gr\"obner basis for ideals) bases, furnish the special fiber with an action by an algebraic torus, opening the door to the world of convex bodies, and methods from combinatorics.    In recent work of Harada-Kaveh \cite{HK} and Nishinou-Nohara-Ueda \cite{NNU}, flat degenerations are also effectively applied to questions of a symplectic nature.  Given a smooth, projective variety $X \subset \mathbb{P}(M)$, with symplectic form induced from the Fubini-Study K\"ahler form on the ambient space, and a flat family $\mathcal{X} \subset \mathbb{P}(M) \times \C$, with $\pi: \mathcal{X} \to \C$, $\pi^{-1}(c) \cong X, c \neq 0$ and $\pi^{-1}(0) = X_P \subset \mathbb{P}(M)$ a toric variety with an action of a torus $T_P$ with momentum polytope $P$, Harada and Kaveh show the following: 

\begin{enumerate}
\item there exists a continuous, surjective map $\Phi: X \to X_P$;\\
\item there exists a dense, open subspace $X^o \subset X$, stable under a compact torus $\mathbb{T}_P \subset T_P$, such that $\Phi|_{X^o}$ is a symplectomorphism (isomorphism of symplectic manifolds) onto $\Phi(X^o)$.
\end{enumerate}

\noindent
 Harada and Kaveh's map $\Phi$ makes the intuition that $X$ and $X_P$ have ``very similar" geometry precise in the symplectic category; indeed, they share a dense, open integrable system on  $X^o$.  Furthermore, it follows that the momentum map of $X^o$ extends continuously to $X$ by way of $\Phi: X \to X_P$.  This construction need not only be applied to toric degenerations, indeed any degeneration to a variety with a torus action will produce a \emph{contraction map}, and a dense, open (not necessarily integrable) Hamiltonian system on $X$. From now on, we call $\Phi$ the contraction map, and we say that a degeneration of an algebraic variety, viewed as a symplectic space, possesses a contraction map when a map to the special fiber exists with the properties above.  Degenerations with contraction maps provide a large family of new integrable systems, and proving that a given dense, open integrable system in a symplectic manifold results from a degeneration is a way to establish that its momenta extend continuously to the ambient space. 

In \cite{HK} and \cite{NNU}, the contraction map is defined by utilizing a powerful construction of Ruan \cite{R}.  Ruan proves existence of continuous maps between hypersurfaces in a K\"ahler manifold by way of a gradient flow.  There are two drawbacks to this method: firstly the resulting map is hard to compute (indeed, the relevant theorem of \cite{HK} on $\Phi$ is mainly an existence result), and secondly, in its present form, the theory applies primarily to smooth projective varieties, and their torus quotients.  The purpose of this paper is to study a family of not necessarily smooth, not necessarily compact varieties which possess flat degenerations with explicitly computable contraction maps. 

In place of smoothness and compactness, we assume our varieties $X$  are semi-projective (i.e. projective over an affine variety) and have an action by a connected complex linearly reductive group $G$ (which is linearized for a relative polarization).  Popov \cite{P} defines a flat degeneration of any such variety to its so-called \emph{horospherical contraction} $X \Rightarrow X^{\hc}.$  The latter is also a $G$-variety, and moreover comes with an additional action by a maximal torus $T \subset G$.  Let $U_-, U_+ \subset G$ be opposite maximal unipotent subgroups adapted to $T$, the horospherical contraction $X^{\hc}$ can be understood as the following GIT (geometric invariant theory)-quotient:

\begin{equation}\label{fromone}
X^{\hc} = \big[X\q U_- \times U_+ \ql G\big]\Big/\!\!\!\!\Big/ T.\\
\end{equation}

\noindent
Here the quotients $X\q U_-$ and $U_+ \ql G$ both come with a residual $T$-action, and the quotient by $T$ is defined through
the diagonal of these actions. 

We study the same gradient flow as \cite{HK} and \cite{NNU}, and show that a contraction map not only always exists, but moreover can be described explicitly.  In fact, it can entirely be formulated in terms of the Hamiltonian geometry of $X$, without referring to the degeneration $X\Rightarrow X^{\hc}$, which is why we refer to it as the \emph{symplectic contraction map}, the image of a which is a new symplectic space, canonically determined by the Hamiltonian geometry of the original $X$.

What allows us to do this is the so-called Vinberg monoid $S_G$ associated to $G$, in particular its Hamiltonian geometry.  The monoid $S_G$ can be understood to be the universal horospherical degeneration of $G$, but as the name indicates its total space is also a monoid -- in fact Vinberg's motivation for introducing it was a universal property in the category of reductive monoids which it exhibits. For our purposes this monoid structure enables us to bring differential-geometric decomposition theorems for Lie groups into play, allowing us to handle the flow much more concretely.

To describe the symplectic contraction, we now fix a maximal compact subgroup $K \subset G$, with maximal torus $\mathbb{T} \subset T$. 
We carry out the horospherical contraction construction in the symplectic category using the Hamiltonian analogue of the quotient by a maximal unipotent subgroup: the symplectic implosion operation of Guillemin, Jeffrey, and Sjamaar, \cite{GJS}.  This operation replaces a Hamiltonian $K$-space $X$ with a (singular) Hamiltonian $\mathbb{T}$-space $EX$.  The group $G$ is replaced by the
cotangent bundle $T^*K$ of the group $K$, which comes with the right $E_{\mathcal{R}}(T^*K)$, respectively left $E_{\mathcal{L}}(T^*K)$ implosions with respect to the right and left actions of $K$ -- see Section \ref{prelim} for more details.  For $X$ a Hamiltonian $K$-space, the symplectic contraction $X^{\sc}$ is defined as the following symplectic reduction:

\begin{equation}\label{fromtwo}
X^{\sc} = \big[EX \times E_{\mathcal{L}}(T^*K)\big]\Big/\!\!\!\!\Big/_{\!\! \! 0}\mathbb{T}.\\
\end{equation}

\noindent
We show that $X^{\sc}$ is naturally homeomorphic to $X^{\hc}$, given its analytic topology.  The contraction
$X^{\sc}$ is a singular Hamiltonian $(K\times \mathbb{T})$-space.    We relate the geometry of $X^{\sc}$ to $X$ with our first result (see Section \ref{gen}). 

\begin{theorem}\label{mainthm}
For a Hamiltonian $K$-space $X$, there is a surjective, continuous, proper map $\Phi_X: X \to X^{\sc}$ which
intertwines the Hamiltonian $K$-actions, in particular $\mu(x) = \mu(\Phi_X(x))$ for any point $x \in X$ -- here $\mu$ denotes the momentum maps for the $K$ action on either $X$ or $X^{\sc}$. 
Furthermore, the map $\Phi_X$ restricts to a symplectomorphism on a dense, open subspace $X_I \subset X$ onto
a $\mathbb{T}$-stable subspace of $X^{\sc}$. 
\end{theorem}

The equivalence relation $\sim_c$ on $X$ defined by the map $\Phi_X$ is described naturally using the momentum map $\mu$. Two points $x, y \in X$ map to the same point in $X^{\sc}$ if and only if $\mu(x) = \mu(y)$ and $kx = y$ for some $k \in [K_{\mu(x)}, K_{\mu(x)}],$ where $K_{\mu(x)} \subset K$ is the stabilizer subgroup of $\mu(x)$.  As a consequence, the space $X^{\sc}$ can be realized topologically as the quotient space $X/\sim_c$.

  The highest face $\Delta_I \subset \Delta$ of the Weyl chamber hit by the momentum map of $X$ is called the principal face, and we call the set $X_I=K\mu^{-1}(\Delta_I) \subset X$ the principal subspace.  This subspace comes by way of the cross-section theorem of Lerman, Meinrenken, Tolman, and Woodward \cite{LMTW}. 
In order to define the action of  an element $t\in \mathbb{T}$ on $X_I$ we first select, for a point $x\in X_I$, an $h \in K$ such that $\mu(h x) \in \Delta_I$.  With this we can then define the following action:

\begin{equation}
t\star x = hth^{-1}x.\\
\end{equation}

\noindent
We prove this action is well-defined using the cross-section theorem in Section \ref{mainsec}.  Theorem~\ref{mainthm}
then shows that this action is Hamiltonian, and that its momentum map $\mu_{\mathbb{T}}$ extends continuously
to $X$.  We also show in Section \ref{mainsec} that the momentum image $\mu_{\mathbb{T}}\circ \Phi_X(X)$ coincides
with the Kirwan polyhedron $\mu(X) \cap \Delta$ of the $K$-action on $X$ (recall that $X^{\sc}$ is a $(K\times\mathbb{T})$-space, here $\mu_{\mathbb{T}}$ is the moment map for the action of $\mathbb{T}$ as the second factor of $K\times \mathbb{T}$).

Going back to horospherical contractions of semi-projective varieties, we then show that the gradient flows associated to them now indeed give rise to the symplectic contractions as in Theorem \ref{mainthm}.  To demonstrate this, we first establish the result (using the Vinberg monoid) for the horospherical contractions of the group $G$ itself.  Using this, we can then extend the result to horospherical contractions of general semi-projective $G$-varieties.

The rest of the paper is devoted to applying our results to the study of branching problems in the representation theory of the group $G$. For a map of connected semi-simple complex groups, $\phi: H \to G$, the branching problem is the computation of the decompositions of the irreducible $G$ representations into irreducible $H$ representations: 

\begin{equation}
M_{\lambda} = \bigoplus_{\eta \in \mathfrak{X}_H} \Hom_H(M_{\eta}, M_{\lambda})\otimes M_{\eta}.\\
\end{equation}

\noindent
The branching problem associated to $\phi$ is geometrically encapsulated in an affine branching variety $X(\phi)$, see Section \ref{branching}. This variety comes with an action by the product of maximal tori $T_H\times T_G$.  The following is proved in \cite{M}:

\begin{theorem}\label{branch}
For every factorization $\pi: H \to F, \psi: F \to G$ of $\phi = \psi \circ \pi$, there is a cone $C(\pi, \psi) = \Delta_F^{\vee}$ of valuations
on the coordinate ring $\C[X(\phi)]$, such that the associated graded algebra of a generic element of this cone is the invariant
ring $\C[X(\psi) \times X(\pi)]^{T_F} = \C[(X(\psi)\times X(\pi))\q T_F].$
\end{theorem}
\noindent We shall use the term \emph{cone} throughout to mean a finitely generated convex polyhedral cone in a real vector space.

For maximal compact subgroups $L \subset H$, $K \subset G$, the space $X(\phi)$ can be realized as a Hamiltonian $(H\times K)$-space.  In particular, $X(\phi)$ comes equipped with a symplectic structure inherited from $E_{\mathcal{R}}(T^*L) \times E_{\mathcal{R}}(T^*K)$ by way of a symplectic reduction, see Proposition \ref{affine-branching-reduction}.
This allows one to study the branching problem using symplectic geometry.  We apply Theorem \ref{mainthm} to prove the following symplectic analogue to Theorem \ref{branch} (see Section \ref{secondrefsection}).

\begin{theorem}\label{mainbranch}
For $\phi: L \to K$ a map of compact semi-simple Lie groups, and a factorization $\pi: L \to J, \psi: J \to K$, there is a corresponding
surjective, continuous, proper contraction map $\Phi_{\pi, \psi}: X(\phi) \to [X(\psi) \times X(\pi)]\q_{\! 0} \mathbb{T}_J$.  This map
restricts to a symplectomorphism on a dense open subset $X^o(\phi) \subset X(\phi)$, identifying this subspace as a Hamiltonian 
$(\mathbb{T}_L \times \mathbb{T}_J \times \mathbb{T}_K)$-space.  The momentum map $\mu_{\mathbb{T}_L \times \mathbb{T}_J \times \mathbb{T}_K}: X^o(\phi) \to \mathfrak{t}_L^* \times \mathfrak{t}_J^* \times \mathfrak{t}_K^*$ extends to a continuous map
on all of $X(\phi),$ with image the cone $P(\psi, \pi) \subset \Delta_L \times \Delta_J \times \Delta_K$ spanned by those triples $(\mu, \eta, \lambda)$
such that $M_{\mu} \subset M_{\eta} \subset M_{\lambda}$.
\end{theorem}

In Section \ref{GTsystem} we use a version of Theorem \ref{mainbranch} (Proposition \ref{chaincontract}) to give a construction of the celebrated Gel'fand-Tsetlin integrable system in a coadjoint orbit $\mathcal{O}_{\lambda}$ of the unitary group $\U(m)$.  This system was first constructed by Guillemin and Sternberg in \cite{GSt} in order to investigate issues of quantization on complex flag manifolds, and it is the focus of Nishinou, Nohara, and Ueda's work in \cite{NNU}.   A similar system was constructed by Hausmann and Knutson \cite{HaKn} on weight varieties of the Grassmannian variety $Gr_2(\C^n)$ and generalized by Howard, Millson, and the second author in \cite{HMM}. This system is shown to be obtained from a flat algebraic degeneration and is related to an integrable system on moduli spaces of Euclidean polygons. In the examples of Section \ref{branching} we explain how our construction produces the systems studied in \cite{HK} and \cite{HMM} and relates these systems to the tropical Grassmannian variety of Speyer and Sturmfels \cite{SpSt}. Several aspects of recent work of Lane \cite{JLane} on the Hamiltonian geometry of Thimm's trick are related to our results.  In particular, one can view our contraction map $\Phi_X: X \to X^{\sc}$ as a continuous extension of Thimm's trick from the principal subspace $X_I \subset X$ to all of $X$.

\subsection*{Acknowledgements} We would like to thank Megumi Harada, Kiumars Kaveh, Frances Kirwan, Allen Knutson and Reyer Sjamaar  for useful conversations.  This project was started during the workshop \emph{Okounkov Bodies and Applications}, organized by Megumi Harada, Kiumars Kaveh and  Askold Khovanskii at the Mathematisches Forschungsinstitut Oberwolfach in May 2014; we are very grateful to the organizers and the MFO for an inspiring program and excellent working conditions. 

 It is our understanding that at least some of the ideas presented here were known to Allen Knutson, who discussed them in a workshop at the Fields Institute in June 2001, and mentioned it in a MathOverflow answer \cite{MOKnutson}.
 
CM was partially supported by NSF Grant DMS-1500966.

\section{Preliminaries}\label{prelim}

\subsection{Notation}

\begin{enumerate}
\item $G$ -- a connected complex linearly reductive group.\\
\item $T \subset G$ -- a maximal algebraic torus of $G$, which we will also denote by $T_G$ when confusion is possible.  Sometimes we will also use the \emph{abstract} maximal torus (cfr. \cite[p. 137]{chrginz}), denoted by $T_G^{\abs}$.  \\
\item $U_-, U_+ \subset G$ -- opposite maximal unipotent subgroups of $G$.  If confusion is possible we will denote the subgroup $U_+$ of a group $H$ by $U_H$.\\ 
\item $K$ -- a connected compact Lie group, a compact form of $G$ when relevant: $K_{\mathbb{C}}=G$.\\ 
\item $\mathbb{T} \subset K$ -- a maximal torus of $K$.\\ 
\item $\mathfrak{k}, \mathfrak{k}^*$ -- the Lie algebra of $K$ and its dual.  We abbreviate the co-adjoint representation of $K$ on $\mathfrak{k}^*$ by denoting $\operatorname{Ad}^*(g)(v)$ as $gv$.  \\
\item $\langle -, -\rangle$ -- a $K$-invariant inner product on $\mathfrak{k}.$\\
\item $\mathfrak{t}, \mathfrak{t}^*$ -- the Lie algebra of $\mathbb{T}$ and its dual.\\
\item $\Delta$, $\Delta^{\vee}$ -- a Weyl chamber and its dual Weyl chamber for $G$ or $K$, determined by the choice of $U_+$.  When necessary we shall indicate the relevant group by $\Delta_G, \Delta_K^{\vee}$ to avoid confusion.\\
\item $\mathfrak{X}_G,\mathfrak{X}_G^+$ -- the lattice of weights of $G$, and its submonoid of dominant weights.\\
\item $M_{\lambda}$ -- the irreducible representation of $G$ (or $K$) associated to a dominant weight $\lambda \in \mathfrak{X}^+_G.$\\
\item $\mathcal{L},\mathcal{R}$ -- actions of a group $G$ or $K$ on itself, given by $\mathcal{L}_g(h)=gh$ and $\mathcal{R}_g(h)=hg^{-1}$.  We shall also denote the induced actions of $K$ on $T^*K$, as well as the extensions of these actions to monoids containing $G$ as group of units, by the same symbols.  We shall occasionally abuse terminology and refer to these actions as the \emph{left}- and \emph{right}-actions, even though all actions occuring in this paper are left actions in the conventional sense.\\
\item $X\Rightarrow X_0$ -- degeneration of an algebraic variety $X$ to $X_0$, i.e. a flat morphism $\pi:\mathcal{X}\rightarrow \mathbb{C}$ such that $\pi^{-1}(1)\cong X$ and $\pi^{-1}(0) \cong X_0$.  The base $X_0$ will typically be an affine toric variety, and the family will be trivial over the main torus orbit (e.g. $\mathbb{C}\setminus\{0\}$ if $X_0=\mathbb{C}$). 
\end{enumerate}

\subsection{Decompositions}
We shall begin by briefly recalling the matrix and Lie group decompositions we shall use.  All of these are standard, see e.g. \cite{Kna, HN}, but we shall particularly use them in the Hamiltonian setting to serve our purposes.  In what follows $G$ will be a connected complex linearly reductive group, with a chosen compact form $K$.
\subsubsection{Polar decomposition and momentum maps}
Let $M$ be a complex finite-dimensional Hermitian vector space.  Then any linear operator $B\in\End(M)$ can be written as $B=UP$, where $U$ is unitary, and $P$ is positive semi-definite.  In this decomposition $P$ is always unique (and equal to $\sqrt{B^*B}$), and if $B$ is invertible then $U$ is unique as well.  

One could equally well write a decomposition $B=\widetilde{P}\widetilde{U}$, with $\widetilde{P}=\sqrt{BB^*}$ positive semi-definite, and $\widetilde{U}$ unitary, but we shall stick to the common convention.  

Moreover, if we equip $\End(M)$ with the K\"ahler form \begin{equation}\label{formEnd}\omega_{\End(M)}(A,B)=-\Im\big(\Tr(AB^*)\big),\end{equation} then we have as momentum map for the $\mathcal{R}$-action of $\U(M)$ on $\End(M)$:
$$\mu(A)=iA^*A,$$ where we have identified $\mathfrak{u}(M)$ and $\mathfrak{u}^*(M)$ through $\langle A,B\rangle=-\Tr(AB)$.  In combination with the polar decomposition this allows us to write down a preferred section $s$ for $\mu$.  Indeed, if $B\in\mu(\End(M))$, i.e. if $-iB$ is  positive semi-definite, then we can just put \begin{equation}\label{section}s(B)=\sqrt{-iB}.\end{equation}

\subsubsection{Cartan decompositions} \label{CartDecomp}
The global Cartan decomposition gives a diffeomorphism
\begin{equation}\label{globalcartan}K\times \mathfrak{k}\overset{\cong}{\longrightarrow} G:(k,x)\mapsto ke^{ix}.\end{equation} 
The Cartan decomposition is compatible with the polar decomposition, in the following sense: for every unitary representation $\phi:K\rightarrow \End(M)$ (which induces a complex representation of $G$, also denoted by $\phi$) we have that 
the polar decomposition of $ke^{ix}$ is given by $$U=\phi(k)\ \ \ \text{ and }\ \ \ P=\phi(e^{ix})=e^{i\, d\phi(x)}.$$
A straightforward corollary of the global Cartan decomposition is the $KAK$-decomposition: every element of a $G$ can be written as $g=k_1ak_2$, where $k_1, k_2\in K$ and $a\in A$, the (real-analytic) abelian connected subgroup of $G$ whose Lie algebra $\mathfrak{a}$ occurs in the Iwasawa decomposition of $G$.  If confusion is possible we shall denote this $A$ factor of $G$ as $A_G$.

\subsection{Symplectic structures on $G$ and $T^*K$}\label{sympG}

As the group $G$ and its symplectic counterpart $T^*K$ play such a central role in our exposition, we summarize some basic properties and conventions about them here.

We choose a maximal compact subgroup $K \subset G$ with maximal compact torus $\mathbb{T} \subset T$. 
We identify $\mathfrak{k}$ with the $\mathcal{L}$-invariant vector fields on $K$, which in turn induces an identification $T^*K\cong K \times \mathfrak{k}^*$. Using this we have \begin{equation*}\mathcal{L}_h(k,v)=(hk,v)\ \ \ \text{ and }\ \ \ \mathcal{R}_h(k,v)=(kh^{-1},hv).\end{equation*}   The momentum maps for the $\mathcal{L}$- and $\mathcal{R}$-actions of $K$ on $T^*K$ (equipped with its canonical symplectic structure) are given by \begin{equation*}\mu_{\mathcal{L}}(k,v)=-kv\ \ \ \text{ and }\ \ \ \mu_{\mathcal{R}}(k,v)=v\end{equation*} respectively.  Note that these actions and their momentum maps are intertwined by the symplectic involution $\iota$, given by $\iota(k,v)=(k^{-1}, -kv)$.

We also want to consider $G$ as a K\"ahler space, with a corresponding equivariant symplectomorphism $G\cong T^*K$.  There are a number of ways that this can be done, e.g. one could choose a $K$-invariant inner product to identify $\mathfrak{k}\cong \mathfrak{k}^*$, and then use the Cartan decomposition (see Section \ref{CartDecomp}) to obtain $$G\cong K\times\mathfrak{k}\cong K\times \mathfrak{k}^*\cong T^*K.$$

However, this is not what we shall use here.  Rather, we shall use identifications $G\cong T^*K$ obtained through  faithful representations $M$ of $G$ that realize $G$ as closed subvarieties of $\End(M)$.  Indeed, as explained in \cite[Appendix A]{MT}, for any symplectic structure on $G$, such that the $\mathcal{L}$- and $\mathcal{R}$-actions are Hamiltonian with momentum maps respectively $\mu_{\mathcal{L}}$ and $\mu_{\mathcal{R}}$, and such that there exists a $(K\times K)$-equivariant projection $\Pi:G\rightarrow K$ whose fibers are Lagrangian, there exists a unique symplectomorphism \begin{equation}\label{GTK}G\cong K\times \mu_{\mathcal{R}}(G)\subset T^*K:g\mapsto(\Pi(g),\mu_{\mathcal{R}}(g))\end{equation} that intertwines the momentum maps for both $\mathcal{L}$- and $\mathcal{R}$-actions.  Moreover, if we choose a faithful representation $G\hookrightarrow \End(M)$ (where $M$ has a $K$-invariant Hermitian product) and restrict the K\"ahler form (\ref{formEnd}) to $G$, the projection $\Pi$ onto the first factor under the global Cartan decomposition (\ref{globalcartan}) has the desired properties.  Finally, if the faithful representation actually realizes $G$ as a closed subvariety of $\End(M)$, then it follows from \cite[Theorem 4.9]{Sj2} and the algebraic Peter-Weyl decomposition of $G$ that $\mu_{\mathcal{R}}:G\rightarrow \mathfrak{k}^*$ is surjective, hence (\ref{GTK}) gives us an identification $G\cong T^*K$.

\subsection{Hamiltonian $K$-spaces}\label{hamiltonianspaces}

The occurrence of the imploded cotangent bundles $E_{\mathcal{R}}(T^*K)$ and $E_{\mathcal{L}}(T^*K)$ requires us to work with general Hamiltonian $K$-spaces.  We let $X$ be a connected Hausdorff topological space with a locally finite decomposition $X = \coprod_{\sigma \in \Sigma} X_{\sigma}$ into connected manifolds, each equipped with a  symplectic form $\omega_{\sigma}.$ A Hamiltonian $K$-action on $X$ is a continuous, decomposition preserving action which is smooth on each $X_{\sigma}$, together with a continuous, equivariant map $\mu: X \to \mathfrak{k}^*$ which restricts to a momentum map on each $X_{\sigma} \subset X$ with respect to $\omega_{\sigma}.$ For the purposes of this paper we will also require that $X$ has a unique top piece $X^o$, such that $X = \overline{X^o}.$

For any Hamiltonian $K$-space $X$, there is an isomorphism $a:  \big( X\times T^*K \big)\q_{\! 0} K \cong X$, cfr. \cite[Lemma 4.8]{GJS}.  Here $K$ acts diagonally on $X\times T^*K$, using the $\mathcal{L}$-action on $T^*K$.  The map $a$ is computed on $\mu^{-1}(0) = \big\{ \left(x, (k, v)\right) \big| \mu(x) = kv\big\}$ by sending $\left(x, (k, v)\right)$ to $k^{-1}x$.  If we also consider the map $b: X \to X \times T^*K$  which sends a point $x$ to $\big(x,(1, \mu(x))\big)$ (and whose image is in the level-set of the momentum map for the diagonal $K$-action), we see that $a$ is the inverse to $b$ composed with the projection map onto $\left(X\times T^*K\right)\q_{\! 0}K$.

Like Hamiltonian manifolds, Hamiltonian $K$-spaces have symplectic reductions, defined to be the topological space $X\q_{\! 0} K = \mu^{-1}(0)/K.$ This space decomposes into the (possibly singular) symplectic reductions $X_{\sigma}\q_{\! 0} K$, which further decompose into manifolds by \cite{SjL}. Each $X_{\sigma}\q_{\! 0} K$ has a unique open, dense top component to its decomposition, so $X\q_{\! 0} K$ possesses a dense open piece, which comes with its reduced smooth symplectic structure.  For a product group $K \times L$,  it is straightforward to check (see \cite[Section 4]{SjL}) that reduction in stages holds for Hamiltonian $(K\times L)$-spaces, in particular $X\q_{\! 0} K$ is a Hamiltonian $L$-space.  Reduction at non-zero levels of the momentum map is performed by the shifting trick $X\q_{\lambda} K = \big(X \times \mathcal{O}_{\lambda}\big) \q_{\! 0} K$, as in the smooth case, where $\mathcal{O}_{\lambda}$ is the co-adjoint orbit through $\lambda$.

We let $\Delta = \cup_{I} \Delta_I$ be the decomposition of the Weyl chamber into relatively open faces with $\Delta^o$ the interior, and we let $K_I \subset K$ be the compact subgroup which fixes all $w \in \Delta_I$. For any Hamiltonian $K$-space we will consider the decomposition by momentum image: for any $I$, let $X_I \subset X$ be the inverse image of $K\Delta_I \subset \mathfrak{k}^*.$ Each of these subspaces has a $K$-action, and can be stratified into $K$-stable manifolds by intersecting $X_I$ with the components of the stratification $X_{\sigma} \subset X$ to obtain $X_{I, \sigma} = X_I \cap X_{\sigma}$ .   For a Hamiltonian $K$-manifold we say that the face $\Delta_I \subset \Delta$ is the principal face if it is the highest face under inclusion with a non-empty intersection with $\mu(X)$.  The principal face is the subject of the Cross Section Theorem, see \cite{LMTW} and \cite{GJS}. 

\begin{theorem}[Cross Section Theorem]
For any Hamiltonian $K$-manifold, $(X, \omega, \mu)$ with principal face $\Delta_I$, the subspace $X_I$ is smooth and dense in $X$.  Furthermore, the group $[K_I, K_I]$ acts trivially on $X_I,$ and $X_I$ is symplectomorphic to $K\times_{K_I} \mu^{-1}(\Delta_I)$. 
\end{theorem}

Though we will not use it, we remark for completeness that the Cross Section Theorem can be extended to all faces, see \cite[Theorem 6.1]{Mein}.
If $X$ is a general Hamiltonian $K$-space, we may apply the Cross Section Theorem to $X^o$, producing $X^o_{ I} \subset X^o$, the principal subspace.  As $X^o_I$ is dense in $X^o$, it follows that $X^o_I$ is dense in $X$.  As a consequence, the momentum image $\mu(X^o_{ I})$ is dense in $\mu(X)$, so for any other face $\Delta_J \subset \Delta$ with $X_J \neq \emptyset,$ we must have $\Delta_J \subset \Delta_I.$  In this sense, a Hamiltonian $K$-space has a principal face, and a version of the Cross Section Theorem holds.

\subsection{GIT and reduction}

Throughout the paper we will work with Hamiltonian $K$-spaces which also have the structure of semi-projective (i.e. projective over an affine) algebraic $G$-varieties, and we need to know that both notions of contraction --    horospherical contractions from (\ref{fromone}) and symplectic contraction from (\ref{fromtwo}) -- agree on such a space.   To be precise, we consider closed, irreducible $G$-varieties $X\subset M\times \mathbb{P}(W)$, where both $M$ and $W$ are $G$-representations equipped with unitary $K$-representation structures. 
The inner products $\langle -, -\rangle$ on $M$ and $W$ induce
a symplectic form $\omega = -\Im\langle-, -\rangle_M+\omega_{\operatorname{FS},W}$ on $M\times \mathbb{P}(W)$, where the latter summand is the Fubini-Study form on $\mathbb{P}(W)$, and a momentum map $\mu_{M\times \mathbb{P}(W)}$.   Note that both projective and affine varieties are special cases hereof.  
Such an $X$ is canonically stratified into smooth, $G$-stable components, each of which inherits a symplectic form and a Hamiltonian $K$-action.  There is a dense, open top component $X^o \subset X$, the complement of the singular locus, hence this equips $X$ with the structure of a Hamiltonian $K$-space. All topological statements should be understood to hold with respect to the analytic topology (though most often are still true for the Zariski topology).

We refer the reader to \cite{Mum} for the definitions of the GIT-quotient $X\q_{\mathbf{L}} G$ of a $G$-variety
with respect to a $G$-linearized relatively ample line bundle $\mathbf{L}.$  We will suppress the linearized bundle $\mathbf{L}$ when
it is clear from context.  We frequently take GIT-quotients of affine varieties $X$ with respect to a torus $T$ and
a choice of character $\eta: T \to \C^*$ which linearizes the action of $T$ (i.e. lifts the action of $T$ on $X$ to the trivial line bundle over $X$ in a possibly non-trivial way).  In this case we write $X\q_{\eta} T.$

For what follows we refer the reader to \cite{SjL} and the book \cite{Mum}.    A point $v \in X$ is said to be analytically semistable ($v \in X^{\an}$) if $\overline{Gv} \cap \mu^{-1}(0) \neq \emptyset$ (here $\mu$ is the restriction of the momentum map for the action of $K$ on $M\times \mathbb{P}(W)$ to $X$).  A point $v \in X$ is said to be algebraically semistable ($v \in X^{\alg}$) with respect to the trivial bundle on $M$, if there is an invariant section which does not vanish at $v$.  For closed, $G$-stable subvarieties $X$, it follows that $X^{\alg} = X \cap \left(M\times \mathbb{P}(W)\right)^{\alg},$ and $X^{\an} = X \cap \left(M\times \mathbb{P}(W)\right)^{\an}$.

For both $X^{\an}$ and $X^{\alg}$ there is a notion of the extended orbit equivalence relation, where $x, y \in X$ are identified if $\overline{Gx} \cap \overline{Gy}\cap X^{\an/\alg} \neq \emptyset.$ We let $X^{\an}\q G$ denote the quotient space by this relation.  The same relation is used to define the GIT-quotient, $X^{\alg}\q G.$  When the momentum map is  \emph{admissible} (see \cite{Sj1}, \cite{Sj2}), the inclusion $\mu^{-1}(0) \subset X^{\an}$ induces a homeomorphism $X^{\an}\q G = \mu^{-1}(0)/K.$ A homeomorphism between the GIT-quotient $X^{\alg}\q G$ and the symplectic reduction $\mu^{-1}(0)/K$ can then be established by showing $X^{\alg} = X^{\an}.$

\subsection{Symplectic implosion}

Our construction of the contraction $X^{\sc}$ makes use of the concept of symplectic implosion, due to Guillemin, Jeffrey, and Sjamaar, \cite{GJS}.  For a Hamiltonian $K$-space $X$, the implosion $EX$ is constructed as the image of $\mu^{-1}(\Delta)$ under the equivalence relation defined by taking the quotient of each subspace $\mu^{-1}(\Delta_I) \subset \mu^{-1}(\Delta)$ by $[K_I, K_I],$ where $K_I$ is the stabilizer (under the co-adjoint representation) of $\Delta_I$.

The implosion(s) of the cotangent bundle $T^*K$ play a universal role in this theory.  We denote the implosion
of $T^*K$ by the right, respectively left $K$-actions by $E_{\mathcal{R}}(T^*K)$ and $E_{\mathcal{L}}(T^*K)$. These are isomorphic Hamiltonian
$(K\times \mathbb{T})$-spaces under the symplectic involution $\iota: T^*K \to T^*K$.  A point in $E_{\mathcal{R}}(T^*K)$ is a pair $(k, w)$
such that $w \in \Delta_I$ and $k \in K/[K_I, K_I].$  The momentum map for the $\mathbb{T}$-action on this space is $\mu_{\mathbb{T}}(k, w) = w.$  A point in $E_{\mathcal{L}}(T^*K)$ is a pair $(k, v)$ such that $-k v \in -\Delta_I \subset -\Delta$, modulo the left action of $[K_I, K_I].$ The residual left $\mathbb{T}$-action on this space has momentum map $\mu_{\mathcal{L}}(k, v) = -k v.$  For any Hamiltonian $K$-space, there are isomorphisms of Hamiltonian $\mathbb{T}$-spaces,

\begin{equation}
EX = K _{\ 0}{\! \ql} \big( X \times E_{\mathcal{R}}(T^*K)\big) = \big(X \times E_{\mathcal{L}}(T^*K)\big)\q_{\! 0} K.\\
\end{equation}

\noindent
For this reason, $E_{\mathcal{R}}(T^*K)$ is referred to as the universal imploded cross-section.  

\begin{remark}In \cite[\S 6]{GJS}, the symplectic implosion $E_{\mathcal{L}}(T^*K)$ is identified with the affine variety $G\q U$, also known as the \emph{basic affine space}, equipped with a suitable K\"ahler metric (to be precise, this is done for $K$ simple and simply connected, but the same argument can be shown to hold in general).  We shall not directly need this identification, see Remark \ref{noneedforGJS6}.
\end{remark}

\section{Horospherical contraction and the Vinberg monoid}

\subsection{Valuations, filtrations, and flat families}

Throughout the paper we will make use of discrete valuations on the coordinate rings of the varieties we consider (for general background regarding this see e.g. \cite[Chapter 9]{AtiMcD}).  For a domain $A$ over $\C$ we require any valuation to satisfy $v(a + b) \geq \max\{v(a), v(b)\},$  with $v(0) = -\infty$ and $v(c) = 0$ for any $c \in \C.$ Any such valuation $v:A \to \Z$ defines an increasing filtration on $A$ by setting $\mathcal{F}^v_{\leq m} = \big\{a\, \big|\, v(a) \leq m\big\} \subset A.$  The filtrations that come from this construction are distinguished by the property that their associated graded algebras $\gr_v(A)$ are also domains -- this is a consequence of the equation $v(ab) = v(a) + v(b).$ 

Moreover, given an increasing filtration $\mathcal{F}$ on a domain $A$ with this property, it is easy to check that $v_{\mathcal{F}}(a) = \min\big\{m\, \big|\, a \in \mathcal{F}_{\leq m}\big\}$ defines a valuation on $A$, and that this construction is inverse to $v \mapsto \mathcal{F}^v.$ With this in mind, we will use the terminology of valuations and filtrations interchangeably. 

Let $v:A \to \Z$ be a valuation as above with $v(A) \subset \Z_{\geq 0}$. The associated Rees algebra $R_v(A) = \bigoplus_{m \in \Z}\mathcal{F}^v_{\leq m}$ is a $\Z$-graded algebra over $\C$ with the following properties: 

\begin{enumerate}
\item $R_v(A)$ is a flat $\C[t]$-algebra, where $t: R_v(A) \to R_v(A)$ maps a graded component $\mathcal{F}^v_{\leq m}$ to its isomorphic copy in $\mathcal{F}^v_{\leq m+1}$;
\item $\frac{1}{t}R_v(A) \cong A[t, \frac{1}{t}]$;
\item $R_v(A)/t R_v(A) \cong \gr_v(A)$ as graded rings.\\
\end{enumerate}

We will be working with semi-projective varieties (recall that these are varieties that are projective over an affine --- both affine and projective varieties being examples).   These can all be characterised as $\Proj \left(\bigoplus_i A_i\right)$, for some graded algebra $\bigoplus_i A_i$ (which will be a domain as the varieties are irreducible), and are projective over $\Spec(A_0)$.  The above discussion carries over to this setting: the valuation will give $\mathcal{F}^v_{i,\leq m}\subset A_i$ for each $i$, and we put $R_v(A)= \bigoplus_{i,m} \mathcal{F}_{i,\leq m}^v$.  We consider this as a graded ring by the $i$-grading, and taking $\Proj$ of this gives us a family over $\mathbb{C}$.  This is a flat degeneration $\Proj\left(\bigoplus_i A_i\right)\Rightarrow \Proj \left(\bigoplus_i \left(\bigoplus_m \mathcal{F}^v_{i,\leq m}/ \mathcal{F}_{i,\leq m-1}^v \right)\right)$.

\subsection{The coordinate rings $\C[G]$ and $\C[G\q U_+]$, and algebraic horospherical contraction}\label{horcontr}

We fix a linearly reductive complex group $G$ with maximal torus ${T}$ and Weyl chamber $\Delta \subset \mathfrak{t}^*.$  We will construct the algebraic horospherical contraction $G^{\hc}$ first, and then use this to construct the contraction $X^{\hc}$ for any semi-projective algebraic variety with a rational $G$-action. The group $G$ is an affine complex variety, and its coordinate ring $\C[G]$ comes with a well-known isotypical decomposition, an algebraic version of the Peter-Weyl theorem: 

\begin{equation}
\C[G] = \bigoplus_{\lambda \in \mathfrak{X}^+_G} M_{\lambda} \otimes M^*_{\lambda}.\\
\end{equation}

Recall that the dominant weights in $\Delta$  come with a partial ordering, where $\eta < \lambda$ when $\lambda - \eta$ can be expressed as a sum of positive roots. Multiplication in the algebra $\C[G]$ is then computed as follows.  First one identifies $M_{\lambda}\otimes M^*_{\lambda}$ with the vector space $\Hom(M_\lambda , M_\lambda)$. 
The tensor product $\Hom(M_{\lambda}, M_{\lambda}) \otimes \Hom(M_{\eta}, M_{\eta})$ is then expanded as follows:

\begin{equation}
\Hom(M_{\lambda}, M_{\lambda}) \otimes \Hom(M_{\eta}, M_{\eta}) = \bigoplus_{\mu, \mu' < \lambda + \eta} \Hom(N_{\lambda, \eta}^{\mu}, N_{\lambda, \eta}^{\mu'}) \otimes \Hom(M_{\mu}, M_{\mu'}).\\
\end{equation}

Here $N_{\lambda, \eta}^{\mu}$ is the multiplicity space for the representation $M_{\mu}$ appearing in the tensor product $M_{\lambda} \otimes M_{\eta}$.  
One projects onto the spaces where $\mu' = \mu,$ then maps each component of this decomposition into $\C[G]$ by sending $\psi \otimes f \in \Hom(N_{\lambda, \eta}^{\mu}, N_{\lambda, \eta}^{\mu}) \otimes \Hom(M_{\mu}, M_{\mu})$ to $\Tr(\psi)f \in  \Hom(M_{\mu}, M_{\mu})$, where $\Tr(\psi)$ denotes the trace.  The multiplicity space $N_{\lambda, \eta}^{\lambda + \eta}$ is $1$-dimensional, 
therefore this direct sum decomposition has a uniquely determined ``top" component, $\Hom(M_{\eta + \lambda}, M_{\eta + \lambda}).$ 

Now we identify the isotypical space $\Hom(M_{\eta}, M_{\eta}) \subset \C[G]$ with $M_{\eta} \otimes M_{\eta}^*$, and we let $U_+ \subset G$ be the maximal unipotent subgroup corresponding to the chosen Weyl chamber $\Delta.$   The (non-reductive) GIT-quotient $G\q U_+ = \Spec(\C[G]^{U_+})$
plays a key role in the description of the horospherical contraction.  The invariants in an isotypical component $M_{\eta} \otimes M_{\eta}^* \subset \C[G]$ are those tensors of the form $v\otimes v_{\eta^*},$ where $v_{\eta^*}$ is the highest weight vector of $M_{\eta}^*.$ As a result, $\C[G]^{U_+}$ is identified with the direct sum of the irreducible representations of $G.$

\begin{equation}
\C[G\q U_+] = \bigoplus_{\eta \in \mathfrak{X}_G^+} M_{\eta}.\\
\end{equation}

\noindent
This algebra is equipped with the Cartan multiplication operation, which is computed on a tensor product by projection onto the highest weight component, $M_{\eta} \otimes M_{\lambda} \to M_{\eta + \lambda}$. In particular $\C[G\q 
U_+]$ is graded by the dominant weights $\eta,$ and has the structure of a rational $G\times T$ algebra.

We now summarize the results of Popov \cite{P} (see also the books of Grosshans \cite{Gr} or Timashev \cite{Tima}, and \cite{M}). One can place a partially ordered filtration on $\C[G]$ by porting over the ordering by the dominant weights on the components $\Hom(M_{\lambda}, M_{\lambda})$.  This is turned into a filtration by non-negative integers by choosing a dominant coweight $h \in \Delta^{\vee}$ in the dual Weyl chamber (we shall moreover choose $h$ to be regular, i.e. to lie in the interior of $\Delta^{\vee}$, though this is strictly speaking not necessary).   The subspace $\mathcal{F}^h_{\leq m} \subset \C[G]$ is then the sum of the $\Hom(M_{\lambda}, M_{\lambda})$ for which $\lambda(h) \leq m.$  Viewing weights $\lambda \in \mathfrak{t}^*$
as functionals, the coweights $h \in \Delta^{\vee} \subset \mathfrak{t}$ are those elements for which $\alpha(h) \geq 0$, for any positive root $\alpha.$ 
It follows that any $h \in \Delta^{\vee}$ gives a linear ordering on $\mathfrak{X}_G$ which respects the partial ordering on dominant weights, and furthermore if $h \in \Delta^{\vee}$ is chosen to be regular then $\lambda < \eta$ implies that $\lambda(h) < \eta(h).$  

\begin{theorem}[Popov]\label{ahcontract}
The filtration $\mathcal{F}^h$ on $\C[G]$ is $G\times G$-stable, with associated graded algebra a domain, in particular each $h \in \Delta^{\vee}$ defines a $G\times G$-stable valuation on $\C[G]$.  When $h \in \Delta^{\vee}$ is chosen to be regular, the associated graded algebra is isomorphic to $\C[G\q U_+ \times U_- \ql G]^{T},$ where the $T$-action is the diagonal action defined through the residual $T \times T$ action on $G\q U_+ \times U_- \ql G$.
\end{theorem}

\noindent
We will denote $\Spec\left(\C[G\q U_+ \times U_- \ql G]^{T}\right)$ as $G^{\hc}$, and refer to it as the horospherical contraction of $G$.  It is straightforward to verify that for a product group, $(G\times H)^{\hc} = G^{\hc} \times H^{\hc}$.

The $G\times G$-stability of the horospherical contraction for $G$ allows a straightforward construction for any semi-projective $G$-variety $X$. 
For $X$ affine, one considers the canonical identification of the coordinate algebra $\C[X]$ with the algebra of $G$-invariants $\C[X \times G]^G,$
where the action by $G$ is defined by the diagonal action on $X$ with the left action on $G.$  If we let $\mathcal{F}_X^h$ be the filtration on $\C[X]$ induced
by the inclusion $\C[X] \subset \C[X \times G]$, the $G\times G$ stability of $\mathcal{F}^h$ gives the following:

\begin{equation}
\gr_{\mathcal{F}_X^h}(\C[X]) = \gr_{\mathcal{F}_X^h}(\C[X \times G]^G) = \big(\C[X] \otimes \gr_{\mathcal{F}^h} \C[G]\big)^G = \C[X \times G^{\hc}]^G.\\
\end{equation}

In this way we obtain Popov's horospherical contraction $X^{\hc} = \Spec( \C[X \times G^{\hc}]^G)$ of an affine variety.  For $X$ semi-projective, with chosen linearization $\mathbf{L}$, one performs the same operation on the homogeneous coordinate ring $R_{\mathbf{L}} = \bigoplus_{m \geq 0} H^0(X, \mathbf{L}^{\otimes m})$:
\begin{equation}\label{generalhorcontr}
X^{\hc} = G \tensor[_{\mathbf{L}}]{\ql}{}\big(X \times G^{\hc}\big) =  \Proj\big(R_{\mathbf{L}} \otimes \C[G^{\hc}]\big)^G.\\
\end{equation} 
\noindent
The quotient variety $G\q U_+$ is universal with respect to GIT  $U_+$-quotients of $G$-varieties in the sense that $X\q_{\mathbf{L}} U_+$ can be computed as $G\tensor[_{\mathbf{L}}]{\ql}{}\big(X \times G\q U_+\big).$  This allows us to compute $X^{\hc}$ by a product $G \times T$-  GIT-quotient, or by bringing in $X\q_{\mathbf{L}} U_-$ and using a GIT-quotient by $T$:
\begin{equation}
X^{\hc} = G \tensor[_{\mathbf{L}}]{\ql}{} \big(X \times G^{\hc}\big) =  G \tensor[_{\mathbf{L}}]{\ql}{} \big(X \times G\q U_- \times U_+ \ql G\big)\q T = \big(X\q_{\mathbf{L}} U_- \times U_+ \ql G\big)\q T\\ 
\end{equation}
Just as in the case $X = G$, the horospherical contraction $X^{\hc}$ has an action of the maximal torus $T$.  The isotypical decomposition of the coordinate ring $\C[X]$ (or $R_{\mathbf{L}}$) is identical to that of the contraction, however the multiplication operation in the coordinate rings of the contractions has an additional grading by dominant weights.  In this way horospherical contraction adds additional torus symmetries to a $G$-variety while maintaining the characteristics of the space preserved under $G$-stable flat degeneration. 
The ability to define the horospherical contraction of a variety by way of the universal properties of the group variety $G$ makes both the definition and computation of this property more tractable (note, howeover, that the horospherical contraction of a variety can also be defined without reference to $G^{\hc}$, by similarly placing a filtration on the coordinate ring of the variety, and then switching to the associated graded).

\subsection{Construction of the Vinberg monoid $S_G$}

Let $G$ be a connected complex linearly reductive group.  The Vinberg monoid $S_G$ canonically associated with it (also known as the \emph{enveloping semigroup of} $G$) is a monoid object in the category of varieties.  Its group of units is the linearly reductive group (\emph{enh} standing for enhanced)$$\Genh=(G\times T_G^{\abs}) /Z_G,$$ where $T_G^{\abs}$ is the (abstract) maximal torus of $G$ and $Z_G$ is the center of $G$ (sitting anti-diagonally in $G\times T_G^{\abs}$).  We will also use the induced compact form $\Kenh=(K\times \mathbb{T}_K)/Z_G$ for $\Genh$, where $\mathbb{T}_K$ is the corresponding compact form of $T_G^{\abs}$.  
The Vinberg monoid can be described as a variety by specifying its coordinate ring $\mathbb{C}[S_G]$ as a subring of $\mathbb{C}[\Genh]$, in particular using the Peter-Weyl decomposition (as a $\Genh\times \Genh$-representation, not as an algebra) of the latter. Indeed, for $\lambda\in\mathfrak{X}_G^+$,  we denote the matrix-coefficients for the irreducible representation with highest weight $\lambda$ as $\mathbb{C}[G]_{\lambda}$. Then we have $$\mathbb{C}[\Genh]=\bigoplus_{(\lambda,\mu)\in \mathfrak{X}^+_{\Genh}} \mathbb{C}[\Genh]_{(\lambda,\mu)} $$ where $\mathfrak{X}_{\Genh}^+$ is the monoid of dominant weights of $\Genh$:
$$\mathfrak{X}^+_{\Genh}=\left\{(\lambda,\mu)\in \mathfrak{X}_G^+\times \mathfrak{X}_G\,\Big|\, \mu-\lambda \in \mathfrak{W} 
\right\},$$  where $\mathfrak{W}=\langle\ \alpha_i\ |\ i=1,\ldots,r\ \rangle$ is the root lattice of G, for $\alpha_i$  the simple positive roots.  
The Vinberg monoid is then defined to be the affine variety with \begin{equation}\label{vinbdef}\mathbb{C}[S_G]=\bigoplus_{(\lambda,\mu)\in \mathfrak{X}_{\Genh}^+\cap Q_G}  \mathbb{C}[\Genh]_{(\lambda,\mu)},\end{equation} where $Q_G$ is the cone in the real vector space $\mathfrak{X}_{\Genh}\otimes_{\mathbb{Z}}\mathbb{R}=\mathfrak{t}^*_{\Kenh}$ given by 
\begin{equation}\label{thecone}Q_G =\left\{ (\lambda,\mu)\in\Delta^{\vee}_{\Kenh}\subset\mathfrak{t}^*_{\Kenh}\, \Big|\, \mu-\lambda=\sum_{i}m_i\alpha_i\, \text{ with all }m_i\in[0,\infty)\right\}. \end{equation}  Vinberg shows \cite{Vi1} that the variety so defined is indeed a monoid with group of units $S_G^{\times}=\Genh$ as described.  

\begin{remark}Note that Vinberg's construction works in general for algebraically closed fields of characteristic zero, and $S_G$ has been constructed for algebraically closed fields of arbitrary characteristic (where the Peter-Weyl decomposition fails to hold) by Rittatore \cite{Rit} using the theory of spherical embeddings.  Though we are not aware of a published version of a more general discussion, the construction should moreover go through for split reductive groups over arbitrary fields.
\end{remark}

Of particular relevance for us is the abelization morphism: by taking the affine GIT-quotient $S_G\ \q \ G\times G$ one obtains an affine space that is characterized as the toric variety for the torus $T_G^{\abs}/Z_G$ with cone the positive Weyl chamber.  We will denote this as $\A_G$ (though it really only depends on the root system of $G$), with morphism $\pi_G:S_G\rightarrow \A_G$.  Vinberg shows that $\pi_G$ is flat, with integral fibers.  The central fiber $\pi_G^{-1}(0)$, a sub-semigroup of $S_G$, was dubbed the \emph{asymptotic semigroup of} $G$ by Vinberg (denoted $\As(G)$), and studied in \cite{Vi2} -- it is nothing more than the horospherical contraction $G^{\hc}$ we already used.

\subsection{Hamiltonian geometry of $S_G$}
We want to consider $S_G$ as a (stratified) symplectic space.  We shall do this by embedding $S_G$ into a matrix space (or rather a direct sum thereof), in the same vein as the discussion for $G$ in Section \ref{sympG}.  This is always possible (see e.g. \cite[Remark, page 169]{Vi1}): take a finite collection of generators $\rho_i$ of the monoid of weights inside $Q_G$ that was used to define $\mathbb{C}[S_G]$ above.  Each corresponding irreducible representation of $\Genh$ extends to all $S_G$, and if we combine them we realize $S_G$ as a closed submonoid of $\bigoplus_i \End(M_{\rho_i})$, equivariant for the $\mathcal{L}$- and $\mathcal{R}$-actions of $\Genh$.  
We can now equip each summand $\End (M_{\rho_i})$ with a K\"ahler structure as in (\ref{formEnd}), which is invariant under $\Kenh$. 
The embedding \begin{equation}\label{embedding}S_G\hookrightarrow \bigoplus_i \End (M_{\rho_i})\end{equation} thus endows $S_G$ with a K\"ahler structure.  
We now  also recall from \cite[Appendix B]{MT} that, for any such K\"ahler structure, we obtain a preferred section $s:\mu(S_G)=\Kenh Q_G\rightarrow S_G$ of the momentum map for the $\mathcal{R}$-action of $\Kenh$ on $S_G$, using the polar decomposition for matrices and the corresponding section as in (\ref{section}).

A direct but important consequence for us of this is that we get a unique symplectic version of the Vinberg monoid:
\begin{lemma}
There exists a stratified symplectic space, unique up to symplectomorphism, such that any choice of generators for the monoid $Q_G\cap \mathfrak{X}_{\Genh}$ induces a symplectomorphism with $S_G$ (with the K\"ahler form induced by the embedding).
\end{lemma}
We shall abuse notation and also refer to this symplectic space as $S_G$.  Remark that the above lemma is restricted to the symplectic structure - the K\"ahler structure is not unique.
\begin{proof}It suffices to write down a canonical symplectomorphism $S_G\rightarrow \widetilde{S_G}$, where $S_G,\widetilde{S_G}$ are two copies of $S_G$ equipped with symplectic structures from different embeddings (\ref{embedding}). By \cite[Theorem 4.9]{Sj2}, the \emph{image} of the momentum map of $\Kenh$ acting (by $\mathcal{R}$) on $S_G$ (though not the momentum map itself) is entirely determined by the weights occuring in $\mathbb{C}[S_G]$.  By the very construction of the Vinberg monoid, this is given by the  union of $\Kenh$-coadjoint orbits of elements in the cone $Q_G$, which is independent of the choice of the $\rho_i$.   Moreover, it also follows from the discussion in \cite[\S 0.6]{Vi1} that $G\times G$ orbits in $S_G$ always get mapped to (coadjoint orbits of) the same faces of $Q_G$.  Hence we can define the desired map using the corresponding sections $s_{S_G}$: $$\Xi: S_G\rightarrow \widetilde{S_G}: x=ks_{S_G}(\mu(x))\mapsto ks_{\widetilde{S_G}}(\mu(x)).$$  This clearly defines a homeomorphism, and we just need to show it is a symplectomorphism.
From this, it suffices to remark that from \cite[\S 0.6]{Vi1} it also clear what the image under this momentum map of $\Genh\subset S_G$ is -- namely, the $\Kenh$ orbit (through the co-adjoint action) of $Q_G$ minus the essential faces other than $O_{\Omega,\Omega}$ (we refer to \cite{Vi1} for this notation).  This in turn, by \cite[Appendix A]{MT}, uniquely maps $\Kenh$ into $T^*\Kenh$.  This shows that $\Xi$ is a symplectomorphism when restricted to $\Genh\subset S_G$.  The rest now follows from continuity.
\end{proof}

Remark that $G\subset S_G$ under this correspondence is $(K\times K)$-equivariantly isomorphic to $T^*K$, from the discussion in Section \ref{sympG}.  

\begin{remark}Note that the properties above -- the existence of an intrinsic symplectic structure, plus a section of the momentum map -- hold true for arbitrary (normal) complex reductive monoids, but we will only be concerned with $S_G$.
\end{remark}

\subsection{Horospherical contraction}
Vinberg's motivation for studying $S_G$ came from the study of reductive monoids, but we will regard $\pi_G:S_G\rightarrow \A_G$ mainly as a degeneration of $G$.  Indeed, it can be regarded as a universal horospherical contraction of $G$:

\begin{proposition}\label{basechange}
Every horospherical contraction of $G$ is induced by $S_G$ through a base change 
\begin{center}
\begin{tikzcd}
\A^1\times_{\A_G} S_G \ar{r} \ar{d}{\pi_h} & S_G\ar{d}{\pi_G}
 \\ \A^1\ar{r} & \A_G,
\end{tikzcd}
\end{center}
in the sense that the family corresponding to the filtration $\mathcal{F}^h$ via the Rees algebra construction is isomorphic to the family $\pi_h$.
\end{proposition}
\begin{proof} Any cocharacter $h:\mathbb{G}_m\rightarrow T_G$ which lies in $\Delta^{\vee}$ induces a monoid morphism $\mathbb{A}^1\rightarrow \A_G$ (sending $0$ to $0$ if $h$ is regular), which we shall use for the base-change.  If we write the generators of the affine coordinate ring $\mathbb{C}[\A_G]$ corresponding to the simple positive roots $\alpha_i$ as $\chi^{\alpha_i}$ then the induced morphism of affine coordinate rings is given by 
$$\mathbb{C}[\chi^{\alpha_i}] \rightarrow \mathbb{C}[t]:\chi^{\alpha_i} \mapsto t^{\alpha_i(h)}.$$
It now suffices to remark that  $$\left(\bigoplus_{(\lambda,\mu)\in \mathfrak{X}^+_{\Genh}\cap Q_G} \mathbb{C}[\Genh]_{(\lambda,\mu)} \right) \bigotimes_{\mathbb{C}[\chi^{\alpha_i}]} \mathbb{C}[t] \ \cong \bigoplus_n \left( \bigoplus_{\lambda(h)\leq n} \mathbb{C}[G]_{\lambda} \right),$$ from which it follows that 
\begin{center}
\begin{tikzcd}
\bigoplus_n \left( \bigoplus_{\lambda(h)\leq n} \mathbb{C}[G]_{\lambda} \right) & \mathbb{C}[S_G]\ar{l}
 \\ \mathbb{C}[t] \ar{u} & \mathbb{C}[\chi^{\alpha_i}] \ar{u} \ar{l} .
\end{tikzcd}
\end{center}
is co-cartesian.
\end{proof}
\begin{remark}
In fact, though we will not use this explicitly, the same is still true if one also looks at \emph{partial} contractions, i.e. those contractions obtained by the same recipe as the horospherical contraction, but using a non-regular $h$.  These would correspond to base-changes of $S_G$ given by monoid morphisms $\A^1\rightarrow \A_G$ that no longer necessarily send $0\in\A^1$ to $0\in\A_G$, but rather to another idempotent in $\A_G$.
Also the rest of the discussion in this paper goes through in this case, utilising the parabolic symplectic implosions discussed by Kirwan in \cite{Ki}.
\end{remark}

\section{The symplectic contraction map}\label{mainsec}

In this section we define the contraction map $\Phi: T^*K \to (T^*K)^{\sc}$, which takes the place of the flat degeneration in Theorem \ref{ahcontract}.  This map is shown to be surjective, continuous, proper, and a symplectomorphism on a dense open subspace. 
We then use a universal property of $(T^*K)^{\sc}$ to construct a contraction map $\Phi_X: X \to X^{\sc}.$

\subsection{The universal contraction $(T^*K)^{\sc}$}

The contraction $(T^*K)^{\sc}$ is the following symplectic reduction:
\begin{equation}
(T^*K)^{\sc} = \big(E_{\mathcal{R}}(T^*K) \times E_{\mathcal{L}}(T^*K)\big)\q_{\! 0} \mathbb{T}.\\
\end{equation}

\noindent
The momentum map $\mu_{\mathbb{T}}: E_{\mathcal{R}}(T^*K) \times E_{\mathcal{L}}(T^*K) \to \mathfrak{t}^*$ of this action takes a pair $\big((k, w),(h, v)\big)$
to the difference $w - hv$.  It follows that $(T^*K)^{\sc}$ is identified with the pairs $\big((k, w),(h, v)\big)$ with $w = h v$, modulo the equivalence relation $$\big((k, w),(h, v)\big) \sim  \big((kt^{-1}, w),(th, v)\big).$$  The space $(T^*K)^{\sc}$ comes with a Hamiltonian $(K\times \mathbb{T} \times  K)$-action.  The $\mathbb{T}$-component of this action is computed as follows: 
\begin{equation}
t \big((k, w),(h, v)\big) = \big((kt^{-1}, w),(h, v)\big) = \big((k, w),(t^{-1}h, v)\big),  \ \ \ \ \ \ \mu_{\mathbb{T}}\big((k, w),(h, v)\big) = w = h v.\\
\end{equation}

\subsection{The map $\Phi$}\label{defcontr}

We now define the map $\Phi_{T^*K}: T^*K \to (T^*K)^{\sc}.$  For a point $(k, v) \in T^*K$, we consider the coadjoint orbit $\mathcal{O}_v\subset \mathfrak{k}^*.$  This orbit intersects $\Delta$ in precisely one point, $w \in \mathcal{O}_v \cap \Delta.$  The coadjoint orbit is isomorphic to $K_I \backslash K,$ where $K_I$ is the stabilizer of $w \in \Delta_I \subset \Delta,$ so we may pick an element $h \in K$ such that $h v = w,$ well-defined up to the left action of $K_I$. We (pre)define $\Phi_{T^*K}$ as follows:
\begin{equation}
\Phi_{T^*K}(k, v) = \big((kh^{-1}, h  v),(h, v)\big) \in E_{\mathcal{R}}(T^*K) \times E_{\mathcal{L}}(T^*K).\\
\end{equation}

\noindent
This map is not well-defined, because the choice of $h$ is unique only up to the $K_I$ action, whereas
the implosion class is modulo $[K_I, K_I].$  However, $(h v) - h  v = 0,$ so the image of $\Phi_{T^*K}$ lies in the momentum pre-image of the diagonal $\mathbb{T} \subset \mathbb{T}^2$ action described above.   This means that we may pass to the quotient $$\Phi_{T^*K}: T^*K \to \big(E_{\mathcal{R}}(T^*K) \times E_{\mathcal{L}}(T^*K)\big)\q_{\! 0} \mathbb{T} = (T^*K)^{\sc}.$$  The space $[K_I, K_I] \backslash K_I$ is covered by the inclusion $\mathbb{T} \subset K_I$. It follows that for any $h \in K_I$, we may find a $t \in \mathbb{T}$ with $t h, h^{-1} t^{-1} \in [K_I, K_I]$.  Therefore $\Phi_{T^*K}$ is well-defined as a map to $(T^*K)^{\sc}.$  By construction, $\Phi_{T^*K}$ intertwines the Hamiltonian $(K\times K)$-action on $T^*K$ with the Hamiltonian $(K\times K)$-action on $(T^*K)^{\sc},$ and preserves the left and right momentum maps.

\begin{proposition}\label{Gcontract}
The map $\Phi_{T^*K}$ is surjective, continuous, and proper. 
\end{proposition}

\begin{proof}

We choose a representative $\big((g, w)(h, v)\big) \in T^*K \times T^*K$ for a point in $\big(E_{\mathcal{R}}(T^*K)\times E_{\mathcal{L}}(T^*K)\big)\q_{\! 0}\mathbb{T}$.   We claim that $(gh, v) \in T^*K$ maps to this point under $\Phi_{T^*K}.$  We may choose any $h' \in K$ such that $h' v = w$, so we choose $h.$ The image $\Phi_{T^*K}(gh, v)$ is then equal to $\big((ghh^{-1}, h v),(h, v)\big) = \big((g, w),(h, v)\big),$ hence it follows that $\Phi$ is surjective. 

We let $C_1 \subset  (T^*K)^{\sc}$ be a compact subset, and we consider the inverse
image $C_2 \subset E_{\mathcal{R}}(T^*K) \times E_{\mathcal{L}}(T^*K)$.  The quotient map is proper, and $\mu^{-1}(0)$ is closed, so $C_2$ is likewise compact.  Now we take another inverse image $C_3 \subset \mu_R^{-1}(\Delta) \times \mu_L^{-1}(-\Delta) \subset T^*K \times T^*K$; as implosion is proper, $C_3 \subset T^*K \times T^*K$ is compact.  The set $C_3$ lies in the set $\mathcal{K} \subset  T^*K \times T^*K$ of $\big((g, w), (h, v)\big)$
with $hv = w \in \Delta$, and is precisely the set of pairs whose equivalence classes lie in $C_1$.  Now we consider the map $a: T^*K \times T^*K \to T^*K: \big((g, w),( h, v)\big) \mapsto (gh , v)$, as in Section \ref{hamiltonianspaces}.  By the surjectivity construction above, the compact image $a(C_3)$ covers $C_1$ under $\Phi_{T^*K}$, and we claim that it coincides with the inverse image of $C_1.$    We suppose that $(k, v) \in \Phi_{T^*K}^{-1}(C_1)$; from this it follows that there is some $h \in K$ with $h v = w$.  We let $g = kh^{-1}$, then by construction $\big((g,w),(h, v)\big) = \Phi_{T^*K}(k, v)$, and $a\big((g, w),( h, v)\big) = (k, v)$. The map $\Phi_{T^*K}$ is therefore proper.

Let $T^*K_{\leq M}$ and $(T^*K)^{\sc}_{\leq M}$ by the subspaces for which $|v| \leq M$, and define variants for $\geq, =$ accordingly. Pick some real number $M > 1$.   By properness $\Phi_{T^*K}$ restricts to a continuous map on $T^*K_{\leq M}$ and $T^*K_{= 1}$, and therefore also on $\R_{> 0}\times (T^*K)_{= 1} \cong T^*K_{> 0}$.  It follows by pasting that $\Phi_{T^*K}$ is continuous on all of $T^*K$. 
\end{proof}

\subsection{$\Phi_{T^*K}$ and sections of momentum maps}
Finally,  it will be useful to re-phrase the map $\Phi_{T^*K}:T^*K\rightarrow (T^*K)^{\sc}$ in terms of sections of the momentum map.  First observe that for $T^*K$  we use a preferred  section of the momentum map for the $\mathcal{R}$ action of $K$ on $T^*K$:

$$s_{T^*K}:\mathfrak{k}^*=\mu_{\mathcal{R}}(T^*K)\rightarrow T^*K: \lambda\mapsto (1,\lambda).$$   This particular choice is compatible with all other choices of sections that we will use throughout the paper, in particular with the section of the $\mathcal{R}$-action of $\Genh$ on $S_G$ that is obtained through the polar decomposition (see the discussion in \cite[Appendix B]{MT}).

We can now do the same for $(T^*K)^{\sc}=\big(E_{\mathcal{R}}(T^*K)\times E_{\mathcal{L}}(T^*K)\big) \q_{\! 0} \mathbb{T}$.  Indeed we simply put 
$$s_{(T^*K)^{\sc}}: \mathfrak{k}^*\mapsto (T^*K)^{\sc}: \lambda\mapsto \left((h^{-1},h\lambda),(h,\lambda)\right),$$
where $h$ is such that $h \lambda$ lies in $\Delta$.  With this in mind we simply have 
\begin{multline}\label{Phithroughsection}
\Phi_{T^*K}:T^*K\cong K\times\mathfrak{k}^* \rightarrow (T^*K)^{\sc}:(k,\lambda)=\mathcal{L}_k s_{T^*K}(\lambda) \\ \mapsto\Phi_{T^*K}\left(\mathcal{L}_ks_{T^*K}(\lambda)\right)=\mathcal{L}_k\Phi_{T^*K}(s_{T^*K}(\lambda))=\mathcal{L}_ks_{(T^*K)^{\sc}}(\lambda).
\end{multline}

\subsection{The general case}\label{gen}

For a Hamiltonian $K$-space $X$ with contraction $X^{\sc},$ we construct a surjective, continuous, proper contraction map $\Phi_X: X \to X^{\sc}.$  We show that $X_I$ is symplectomorphic to $X_I^{\sc}$ when $\Delta_I$ is the principal face of $X$, placing a Hamiltonian $(K\times \mathbb{T})$-structure on $X_I \subset X.$ In the introduction we defined the contraction $X^{\sc}$ of $(X, \omega, \mu)$ as the diagonal symplectic reduction of the product $EX \times E_{\mathcal{L}}(T^*K)$ by $\mathbb{T}$.  This makes $X^{\sc}$ into a Hamiltonian $(K \times \mathbb{T})$-space, and it shows that the image of the $\mathbb{T}$-momentum map coincides with the momentum image for the residual $\mathbb{T}$-action on $E X$, which by \cite{GJS} is the set $\mu(X) \cap \mathfrak{t}^*.$  We connect $X$ to $X^{\sc}$ with the following map:
\begin{equation}
\Phi_X: X \to X^{\sc}: x \mapsto \big((hx),(h, \mu(x))\big).
\end{equation}

\noindent
Here $h \in K$ is chosen to ``diagonalize" $\mu(x),$ that is $h \mu(x) \in \Delta.$  If $h\mu(x)$ actually lies in the face $\Delta_I$, then any two elements $h,h'$ which do this job differ
by an element $g \in K_{I}$, $h = gh'$.  For any such $g$ we may find $t \in \mathbb{T}$ and $p \in [K_I, K_I]$
such that $g = tp$.  Combining this with the equivalence relations which define $X^{\sc}$, we have $\big(hx,(h, \mu(x))\big) = \big(gh'x,(gh', \mu(x))\big) = \big(tph'x,(tph', \mu(x))\big) = \big(ph'x,(ph', \mu(x))\big) = \big(h'x,(h', \mu(x))\big)$, so $\Phi_X$ is well-defined.

\begin{proposition}\label{extension}
For any connected Hamiltonian $K$-space $(X, \omega, \mu)$, the map $\Phi_X: X \to X^{\sc}$ is surjective, continuous, proper, and $K$-equivariant. 
\end{proposition}

\begin{proof}
As defined, $X^{\sc}$ is the symplectic reduction $\big(EX \times E_{\mathcal{R}}(T^*K)\big)\q_{\! 0}\mathbb{T}$, but it can also be constructed as follows:
\begin{equation}
X^{\sc} = \big(EX \times E_{\mathcal{R}}(T^*K)\big)\q_{\! 0}\mathbb{T} = \Big(K _{\ 0}{\!\ql}{} \big(X \times E_{\mathcal{L}}(T^*K)\big) \times E_{\mathcal{R}}(T^*K)\Big)\q_{\! 0} \mathbb{T} = K _{\ 0}{\!\ql}{} \big(X \times (T^*K)^{\sc}\big).
\end{equation}
\noindent
This follows from the universal property of the imploded cotangent bundle, $EX = K _{\ 0}{\!\ql}{} \big(X \times E_{\mathcal{L}}(T^*K)\big).$ Now we can use
the $(K\times K)$-equivariance of $\Phi_{T^*K}: T^*K \to (T^*K)^{\sc}$ to construct $\Phi_X$:
\begin{equation}\label{PhiX}
\Phi_X: X \cong K _{\ 0}{\!\ql}{} \big(X \times T^*K\big)  \to K _{\ 0}{\!\ql}{} \big(X \times (T^*K)^{\sc}\big) \cong X^{\sc}.\\
\end{equation}
\noindent
The stated properties of $\Phi_X$ now follow from the relevant properties of $\Phi_{T^*K}.$
\end{proof}

Next we show that $\Phi_X$ defines a symplectomorphism on a $(K\times \mathbb{T})$-stable subspace of $X.$ 

\begin{proposition}\label{Xsymp}
For a Hamiltonian $K$-space $(X, \omega, \mu)$, let $\Delta_I \subset \Delta$ be the principal face of $\mu(X).$  The map $\Phi_X$ restricts to a symplectomorphism $\Phi_X: X_{ I}^o \to (X^o_I)^{\sc}.$ 
\end{proposition}

\begin{proof}
First we prove that $\Phi_X$ restricts to a diffeomorphism on $X_{I}^o.$ We consider the space $Y = \mu^{-1}(\Delta_I) \times K \subset X_I^o \times K$, and the following commutative triangle:  

\begin{figure}[H]

$$
\begin{xy}
(0, 32.5)*{Y} = "A";
(3.5, 30)*{} = "A1";
(-3.5, 30)*{} = "A2";
(20, 0)*{\big(X_{I}^o\big)^c.} = "B";
(20,5)*{} = "B1";
(15, 1.5)*{} = "B2";
(15, -1.5)*{} = "B3";
(-20, 0)*{X_{I}^o} = "C";
(-20,5)*{} = "C1";
(-15, 1.5)*{} = "C2";
(-15, -1.5)*{} = "C3"; 
"B1"; "A1";**\dir{-}? >* \dir{>};
"C1"; "A2";**\dir{-}? >* \dir{>};
"B2"; "C2";**\dir{-}? >* \dir{>};
"C3"; "B3";**\dir{-}? >* \dir{>};
(0, 4.1)*{\Phi_X|_{X_I^o}};
(0, -3.5)*{a};
(15 , 17)*{\pi_2};
(-15 , 17)*{\pi_1};
\end{xy}
$$\\
\label{Label1}
\end{figure}

\noindent
Here, $\pi_1(y, k) = k y,$ and  $\pi_2(y, k) = \big(y, (k^{-1}, k \mu(y))\big)$, and $a$ is as in Section \ref{hamiltonianspaces}.  Both $\pi_1$ and $\pi_2$ are submersions
onto $X^o_I$ and $(X^o_I)^{\sc}$ respectively, and both maps are constant on the other's fibers.  The commutativity of the diagram
then implies that $\phi$ and $a$ are differentiable inverses to each other.

Now we identify $Y$ with the subspace $Y_0 \subset X \times T^*K$ of those pairs of the form $\big(y, (g^{-1}, g  \mu(y))\big)$.
As constructed, $Y_0 \subset \mu^{-1}(0)$ with respect to the $K$ action on $X \times T^*K.$  Recall that with the standard
 symplectic form $\omega_0$ on $T^*K$ we have $X \cong K _{\ 0}{\!\ql}{} \big(X \times T^*K\big)$, so it follows that
$a^*(\omega) = (\omega + \omega_0) |_{Y_0}.$  Now, the symplectic form $\sigma$ on $(X^o_I)^{\sc}$ satisfies $\pi_2^*(\sigma) =  (\omega + \omega_0) |_{Y_0} = a^*(\omega|_{X^o_I}).$
\end{proof}

When $X = T^*K$, the open subset of Proposition \ref{Xsymp} is $(T^*K)^o \subset T^*K$ consisting of those pairs $(k, v)$
with $v \in K\Delta^o.$  The continuous map $\mu_{\mathbb{T}} \circ \Phi_X: X \to \mathfrak{t}^*$ restricts to give Hamiltonians for the induced $\mathbb{T}$-action on $X^o_I \subset X,$ and it follows from Proposition \ref{extension} that the image of this map is $\mu(X) \cap \mathfrak{t}^*.$ 
The residual $\mathbb{T}$-action on $X^{\sc} $ is computed as follows:
\begin{equation}
t \big(x,(h, v)\big)= \big(tx,(h, v)\big) = \big(x,(t^{-1}h, v)\big).
\end{equation}
\noindent
Given $x \in X_{I}^o,$ with $\Phi_X(x) = \big(hx,(h, \mu(x))\big)$, it follows that the $\mathbb{T}$-action is computed as below:

\begin{equation}
t \star x = a\bigg(t \big( hx,(h, \mu(x))\big)\bigg) = a\big(thx,(h, \mu(x))\big) = h^{-1}th x.
\end{equation}

\section{Gradient flows on the Vinberg monoid}\label{gradientflowmonoid}
\subsection{The work of Harada and Kaveh}
In \cite{HK}, Harada and Kaveh study a flow on the total space of a toric degeneration of a smooth projective variety. 
In particular, given a flat proper algebraic morphism $\pi:\mathcal{X}\rightarrow \mathbb{C}$ (where we denote $\pi^{-1}(0)$ as $X_0$ and  $\pi^{-1}(1)$ as $X$), they consider the \emph{gradient-Hamiltonian vector field} \begin{equation}\label{grad-hamil}V_{\pi}=-\frac{\nabla(\Re(\pi))}{|\!| \nabla(\Re(\pi))|\! |^2},\end{equation} inspired by an earlier work of Ruan \cite{R} (see also \cite{NNU}).  We will always assume that $\mathcal{X}$ is a family of semi-projective varieties as before, and then use the K\"ahler metric we have chosen to define $V_{\pi}$.   They use $V_{\pi}$ to obtain a surjective continuous (in the analytic topology) map $\phi:X\rightarrow X_0$ which extends the flow of $V_\pi$ at time $t=1$ where this is defined, and which is a symplectomorphism restricted to a dense open subset of $X$, such that the integrable system on this subset (thought of as a subset of $X_0$ now) extends continuously to all of $X$.  To be precise, in \cite{HK} Harada and Kaveh establish an existence result for $\phi$, where the continuity depends on the (real-)analicity of the gradient-Hamiltonian vector field $V_{\pi}$.  They use the compactness of $X$ (or more precisely the properness of $\pi$) in their proof.

We are now concerned with horospherical degenerations of a $G$-variety, for $G$ a linearly reductive group.   As described above, such horospherical degenerations can all be understood through the Vinberg monoid $S_G$, and its base-change giving the horospherical degenerations of $G$ itself, which in turn can be used to study arbitrary horospherical degenerations of semi-projective varieties.  This differs from the set-up of \cite{HK} in two ways:
\begin{itemize}
\item the variety we start from is not necessarily projective nor smooth;
\item though the central fiber of our degeneration picks up extra torus symmetry, the dimension of the torus is in general too low to form a toric variety.
\end{itemize}
Nevertheless, we shall follow the same path (to be precise, none of the relevant results on the gradient-Hamiltonian vector field in \cite{HK}, in particular the existence of $\phi$, need the central fiber to be toric).  

Moreover, we shall show that for horospherical degenerations the map $\phi$ not only exists, but can explicitly be realized as the map $\Phi_X$ from before.  The crucial ingredient is the use of the Vinberg monoid, which in particular is both a universal horospherical degeneration of $G$, and a monoid (in fact a universal monoid, see \cite{Vi1}).  It is the last property which allows us to bring decomposition theorems of a differential geometric nature into play, which in turn make the flow for the vector field (\ref{grad-hamil}) transparent and tractible.  

We will begin by studying the flow for $G$ itself, first for the basic example of $G=\SL(2,\mathbb{C})$, and then for the general case where $G$ is linearly reductive (note that the situation simplifies for semi-simple $G$, as explained in Section \ref{ssonly} below).
\subsection{Basic example}
For $G=\SL(2,\mathbb{C})$, the Vinberg monoid is simply given by $S_G=M_{2\times 2}(\mathbb{C})$, with the abelization map $\pi$ given by the determinant $\det: M_{2\times 2}(\mathbb{C})\rightarrow \mathbb{C}$.  Recall that the K\"ahler structure we are using on $M_{2\times 2}(\mathbb{C})$ is given by the Hermitian structure $\langle A, B\rangle=\Tr(AB^*)$.  The group of units $\Genh$ is just $\GL(2,\mathbb{C})$.  We want to consider the gradient-Hamiltonian vector field $V_{\pi}$ given by (\ref{grad-hamil}).  We will see that this simplifies a lot using the decompositions we have mentioned earlier.  In fact, as we want to re-use the results of the next two lemmas later on, we shall at once state them for arbitrary matrix spaces $M_{n\times n}(\mathbb{C})$, with for the rest of this section $K$ referring to $\SU(n)$.  We shall continue to refer to the determinant map $M_{n\times n}(\mathbb{C})\rightarrow \mathbb{C}$ as $\pi$.
\begin{lemma}\label{eerste}
The gradient-Hamiltonian vector field $V_{\pi}$ is invariant under the $\mathcal{L}$ and $\mathcal{R}$-actions of $K$.  
\end{lemma}
\begin{proof}Perhaps the easiest way to see this is to use the fact that $\nabla(\Re\pi)$ is also the Hamiltonian vector field associated with $\Im(\pi)$, see \cite[\S 2.2]{HK}.  Since $\pi$  and the symplectic form are  invariant under both actions of $K$, so is $\nabla(\Re\pi)$ and hence $V_{\pi}$.
\end{proof}
In particular this implies that we can just limit ourselves to understanding the flow of $V_{\pi}$ on the $A$-part of $\GL(n,\mathbb{C})$ -- of course for $G=\SL(2,\mathbb{C})$ we have that $\GL(2,\mathbb{C})=\Genh$.

\begin{lemma}\label{onlytorus}
If $x\in \SL(n,\mathbb{C})\subset \GL(n,\mathbb{C})$ has a $KAK$-decomposition $x=k_1ak_2$ (with $k_1,k_2\in K$), then the flow-line for $V_\pi$ through $x$ is contained in $k_1A_{\GL(n,\mathbb{C})} k_2$.
\end{lemma}
\begin{proof}By Lemma \ref{eerste}, it suffices to show this for the case $k_1=k_2=1$, where it follows trivially.  Indeed, $A_{\GL(n,\mathbb{C})}$ can be taken to be the invertible diagonal matrices with positive entries on the diagonal, of the form $D=\left(\begin{array}{ccc} x_1 & & 0 \\ & \ddots & \\ 0 & & x_n\end{array}\right)$, with all $x_i>0$.  Given such a diagonal matrix $D$, we find straightforwardly $V_\pi(D)=\frac{1}{\sum_i \prod_{i\neq j} x_j^2} \left(\begin{array}{ccc} -x_2\cdots x_n & &  0 \\ & \ddots &\\ 0 & & -x_1\cdots x_{n-1} \end{array}\right).$
\end{proof}

Re-focusing on $G=\SL(2,\mathbb{C})$, the level sets of $\pi$ restricted to $\Aenh=\left\{\left(\begin{array}{cc} x & 0 \\ 0 & y \end{array}\right)\Bigg| x,y>0 \right\}$ consist of single leaves of the hyperbola: $xy=\lambda$, and the vector field $V_{\pi}$ is perpendicular to these leaves; in particular its flow lines will form leaves of hyperbolas $x^2-y^2=c^2$.  It is now elementary to see what they flow to (see Figure \ref{mainfigure}), as we know that the flow $V_{\pi}$ goes from $\pi^{-1}(1)$ to $\pi^{-1}(0)$ in time $1$:

$$\left(\begin{array}{cc} x & 0 \\ 0 & x^{-1} \end{array}\right)\mapsto \begin{cases} 

\left(\begin{array}{cc}\sqrt{x^2-x^{-2}} & 0 \\ 0 & 0 \end{array}\right)
& \text{if }x\geq1; \\ \ & \ \\ 
\left(\begin{array}{cc} 0 & 0 \\ 0 & \sqrt{x^{-2}-x^2} \end{array} \right) & \text{if }x\leq 1.
\end{cases}
$$

\begin{figure}[h!]
\begin{center}
\begin{tikzpicture}
  \tikzstyle{axes}=[thick]

\begin{scope}[scale=2]
  \draw[style=help lines,step=1cm] (-1.4,-1.4) grid (5.4,2.4);
    \begin{scope}[style=axes]
	\draw[->] (-1.5,0) -- (5.5,0) node[below] {$x$};
    	\draw[->] (0,-1.5) -- (0,2.5) node[left] {$y$};
    \end{scope}

    \draw[blue,style=thick] plot[variable=\t,samples=1000,domain=0.4:5.1] ({\t},{1/\t});
        \draw[blue] plot[variable=\t,samples=100,domain=0.7:5.1] ({\t},{2/\t});

        \draw[red] plot[variable=\t,samples=100,domain=-40:60] ({1.1*sec(\t)},{1.1*tan(\t)});
    \draw[red] plot[variable=\t,samples=100,domain=-25:40] ({2*sec(\t)},{2*tan(\t)});
        \draw[red] plot[variable=\t,samples=100,domain=-40:60] ({tan(\t)},{sec(\t)});

     \draw[red](-1,-1) -- (2,2);

\draw[blue] (0.4, 2.5) node[anchor=south] {$\scriptstyle{xy=1}$};
\draw (1.35, 0.82) node[anchor=west] {$\scriptstyle{(x,x^{-1})}$};
\draw (1,0.92) node[anchor=north] {$\scriptstyle{(1,1)}$};
\draw (0.81,1.325) node[anchor=east] {$\scriptstyle{(\tilde{x},\tilde{x}^{-1})}$};

\draw (-0.05,0) node[anchor=north west] {$\scriptstyle{(0,0)}$};
\draw (-0.05,0) node[anchor=north west] {$\scriptstyle{(0,0)}$};
\draw (1.05,0) node[anchor=north west] {$\scriptstyle{(\sqrt{x^2-x^{-2}},0)}$};
\draw (0,1) node[anchor=north east] {$\scriptstyle{(0,\sqrt{\tilde{x}^2-\tilde{x}^{-2}})}$};

\draw[->] (0.775, 1.2) .. controls (0.7, 1.15) and (0.5,.95) .. (0.02,.95);
\draw[->] (.9,.96 ) -- (0.03,0.09);
\draw[->] (1.275,0.75) .. controls (1.25,0.7) and (1.05,0.5) .. (1.05,0.02);

\end{scope}
\end{tikzpicture} 
\caption{\label{mainfigure} \emph{The gradient-Hamiltonian flow on $\Aenh$ for $G=\SL(2,\mathbb{C})$.}}
\end{center}
\end{figure}
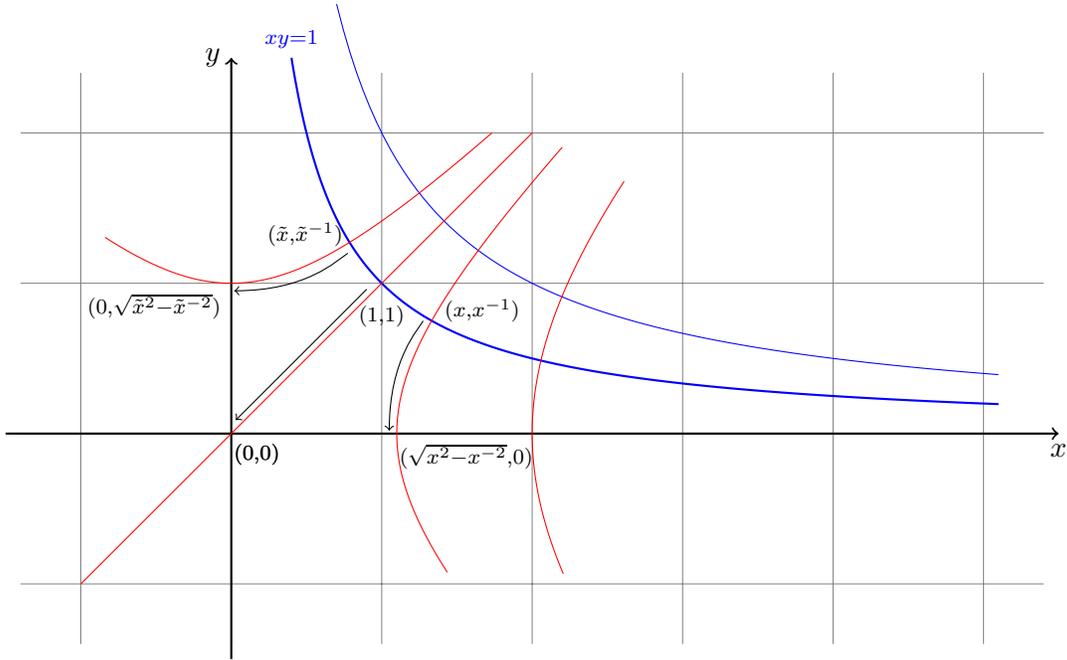

More intrinsically, we can describe the map (on all of $\SL(2)$) using the polar decomposition: 
 \begin{equation}\label{contrSL2}\SL(2,\mathbb{C})\ni B=UP=U\sqrt{B^*B}\mapsto U \sqrt{B^*B-\lambda I},\end{equation} where $\lambda$ is the smallest eigenvalue of $B^*B$.
 
 We remark two things: firstly, the momentum maps for the ${\mathcal{L}}$- and ${\mathcal{R}}$-actions of $K=\SU(2)$ are left invariant by the flow.
Secondly, the stabilizers of these actions remain the same, unless $B^*B=1$.  From this it follows that the map $\Phi$ on $\SL(2,\mathbb{C})$ simply collapses $\SU(2)\subset \SL(2,\mathbb{C})$ (which one can think of as the zero-section of $T^*\SU(2)\cong \SL(2,\mathbb{C})$).

\subsection{General linearly reductive case}

We will now discuss the general case  when $G$ is linearly reductive.  

Roughly speaking, we now want to follow the same strategy as in the $\SL(2,\mathbb{C})$ case for the whole embedding $$S_G\hookrightarrow \bigoplus_i \End(M_{\rho_i}) \hookrightarrow \End\left(\bigoplus_i M_{\rho_i}\right).$$
Note that the situation in general differs in two aspects: firstly, the embedding depends on the choices of the $\rho_i$'s, and in general no canonical choice can be made.
Secondly, the actual degeneration of $G$ we want to consider also involves a further choice that gives us the base-change as in Proposition (\ref{basechange}).

At this point we will assume that the cocharacter $h$ as used in Theorem \ref{ahcontract} is primitive in the coweight lattice, to ensure that we can think of $\mathbb{C}\times_{\A_G}S_G$ as a subvariety of $\bigoplus \End(M_{\rho_i})$ (or $\End(\bigoplus_i M_{\rho_i})$), and thus obtain a K\"ahler structure.

\begin{lemma}\label{comdiagbasechange}With the conventions as above, we have a commutative diagram
\begin{center}
\begin{tikzcd}
\mathbb{C}\times_{\A_G}S_G \ar{r} \ar{d}{\pi_h} &\End\left(\bigoplus_i M_{\rho_i}\right)\ar{d}{\det} \\ \mathbb{C}\ar{r} &\mathbb{C},\\
\end{tikzcd}
\end{center}
where the bottom row is given by a monoid morphism, and hence is given by \begin{equation}\label{int-red}z\in \mathbb{C}\mapsto z^n\end{equation} for some non-negative $n$.
\end{lemma}
\begin{proof} First note that, since $G$ is a closed connected linearly complex reductive subgroup of $\End(\bigoplus_i M_{\rho_i})$, it is necessarily contained in $\SL(\bigoplus M_{\rho_i})$.  It  now suffices to remark that both $\End(\bigoplus_i M_{\rho_i})\rightarrow \mathbb{C}$ and $\mathbb{C}\times_{\mathbb{A}_G} S_G\rightarrow \mathbb{C}$ are categorical quotients for the action of $\SL(\bigoplus_i M_{\rho_i})\times \SL(\bigoplus_i M_{\rho_i})$ and $G\times G$ respectively, and that the intersection of the former orbits with $\mathbb{C}\times_{\mathbb{A}_G}S_G$ are exactly the latter orbits.
\end{proof}

\subsubsection{Contraction of $G$ onto $\As(G)$}
In spite of these complications we shall now show that the basic strategy still goes through, provided that we calibrate the vector fields $V_{\det}$ on $\End(\bigoplus M_{\rho_i})$ appropriately, by using the corresponding integer $n$.    To this end we introduce other normalizations of gradient-Hamiltonian vector fields, labeled by positive integers $m$, by defining (for a flat morphism $\pi:\mathcal{X}\rightarrow \mathbb{C}$ as before) \begin{equation}V_{\pi, m}=-\frac{\nabla(\Re(\pi))}{|\!| \nabla\left(\Re(\pi)\right) |\!|^2} \left(m\left(\Re(\pi)\right)^{1-\frac{1}{m}} \right).\end{equation}
Observe that $V_{\pi, 1}$ gives back the original $V_{\pi}$ as in (\ref{grad-hamil}).  The effect of this is that, on $\pi^{-1}(\mathbb{R})\subset \mathcal{X}$, the flow of $V_{\pi, m}$ at time $t$ is the same as the flow of $V_{\pi}$ at time $t^m$.  Indeed, while the original normalization was chosen to satisfy $V_{\pi}(\Re(\pi))=-1$, we now have $$V_{\pi, m}(\Re(\pi))=-m\left(\Re(\pi)\right)^{1-\frac{1}{m}},$$  from where it follows that, if $f_{m,x}(t)$ is the flow-line through a point $x\in\pi^{-1}(1)$ determined by $V_{\pi, m}$, then $$\pi\circ f_{m,x}(t)=(-t+1)^m.$$
In particular, all of these $V_{\pi, m}$ have the same contraction map $\pi^{-1}(1)\rightarrow \pi^{-1}(0)$, if it exists.

A direct verification also gives
\begin{lemma}\label{compvectfields}
For any $\pi:\mathcal{X}\rightarrow \mathbb{C}$ as before, we have that on $\pi^{-1}(\mathbb{R})$ $$V_{\pi}=V_{\pi^m, m}$$ wherever these vector fields are defined.
\end{lemma}
We have 
\begin{proposition}\label{combinedflow}
If we equip  $\End(\bigoplus_i M_{\rho_i})$ with the vector field $V_{\det, n}$, then the resulting vector field  is tangential to $\pi_h^{-1}(\mathbb{R})$, where it moreover coincides with $V_{\pi_h}$.
\end{proposition}
\begin{proof} It suffices to remark that, under the embedding $\pi_h^{-1}(\mathbb{R})\subset \mathbb{C}\times_{\A_G} S_G\hookrightarrow \End(\bigoplus M_{\rho_i})$, the vector field $V_{\pi_h}$ corresponds to $V_{\det, n}$ by Lemma \ref{compvectfields} and (\ref{int-red}).
\end{proof}

We can finally conclude with
\begin{corollary}\label{flowresult}
The flow at time $t=1$ of $V_{\pi_h}$ extends to a contraction map $\Phi_f$ from $G$ onto $\As(G)$.
\end{corollary}
\begin{proof} 
It suffices to remark that the proofs of Lemmas \ref{eerste} and \ref{onlytorus} are still valid for each $\End(M_{\rho_i})\rightarrow \mathbb{C}$, so we get a contraction map from each $\SL(M_{\rho_i})$ to the corresponding singular elements of $\End(M_{\rho_i})$.  Recall also that $G$ is contained in $\SL(\bigoplus_i M_{\rho_i})$.  The rest now follows from Proposition \ref{combinedflow}.
\end{proof}

\subsubsection{Semi-simple situation}\label{ssonly}
When $G$ is semi-simple, we can also describe the situation slightly differently.  
The main relevance for restricting to the semi-simple case is the following:
\begin{lemma}
If $G$ is semi-simple, then the embedding (\ref{embedding}) gives rise to the commutative diagram
\begin{center}
\begin{tikzcd}
S_G \ar{r} \ar{d} & \bigoplus \End(M_{\rho_i}) \ar{d}{(\det_i)_i} \\
\A_G\ar{r} & \mathbb{C}^r,
\end{tikzcd}
\end{center}
where $r$ is the number of $\rho_i$ chosen.
\end{lemma}
\begin{proof} Recall that each $M_{\rho_i}$ is an irreducible representation of $\Genh$.  We have that $G\subset S_G^\times$, being semi-simple, gets mapped to $\SL(M_{\rho_i})\subset \End(M_{\rho_i})$, and by Schur's lemma $T_G^{\abs}$ acts diagonally on $M_{\rho_i}$.  Therefore $\det_i$ will be trivial on $Z_G$, and descend to $T_G^{\abs}/Z_G$.  We hence find an induced morphism $T_G^{\abs}/Z_G\rightarrow (\mathbb{C}^*)^r$.  That this extends to $\A_G\rightarrow \mathbb{C}^r$ now just follows from the fact that all  weights $\rho_i$ are contained in the cone $Q_G$.
\end{proof}
 All together we have the diagram

\begin{center}
\begin{tikzcd}
\mathbb{C}\times_{\A_G}S_G \ar{r} \ar{d}{ \pi_h} & S_G \ar{r} \ar{d}{\pi_G} & \bigoplus \End(M_{\rho_i}) \ar{d}{(\det_i)_i}  \\ \mathbb{C}\ar{r} \ar[bend right]{rr}{\Psi} &
\A_G\ar{r} & \mathbb{C}^r \\
\end{tikzcd}
\end{center}
In this diagram all maps are monoid-morphisms, and in particular $\Psi$ gives us $r$ positive integers $n_i$, such that \begin{equation}\label{numbers}\Psi(z)=(z^{n_i})_i,\end{equation} and $n=\sum_i n_i$, where $n$ is as in Lemma \ref{comdiagbasechange}.

\begin{lemma}
On $\pi_h^{-1}(\mathbb{R})$, we have that $(V_{\det_i,n_i})_i$ and $V_{\det,n}$ coincide, and hence give rise to the same contraction map $G\rightarrow\As(G)$.
\end{lemma}
\begin{proof}As in the proof of Proposition \ref{combinedflow}, for each of the constituent maps $\phi_i:S_G\mapsto \End(M_{\rho_i})$ of (\ref{embedding}) we have that $d\phi_i\left(V_{\pi_h}|_p\right)=V_{\det_i}|_{\phi_i(p)}$, for any $p\in\pi_h^{-1}(\mathbb{R})$.  From this it follows that $V_{\pi_h}$ and $\left( V_{\det_i,n_i}\right)_i$ coincide on $\pi_h^{-1}(\mathbb{R})$.  The rest follows from Proposition \ref{combinedflow}.  
\end{proof}

\subsubsection{Identification of $\As(G)$ with $(T^*K)^{\sc}$}

We now want to show that (for general reductive $G$) the contraction map from $G$ to $\As(G)$ is the same as the map $\Phi$, when a suitable identification of $\As(G)$ and $(T^*K)^{\sc}$ is made.  A key tool for this is a section of the momentum map.  Recall from (\ref{Phithroughsection}) that we can use a preferred section $s_{(T^*K)^{\sc}}$ of the momentum map $\mu_{\mathcal{R}}$ of the $\mathcal{R}$-action of $K$ on $(T^*K)^{\sc}$ to determine the map $\Phi$.  We will now show that the same is true for the contraction map $\Phi_f$ coming from the gradient-Hamiltonian flow.  Using this we will then show that $\Phi$ and $\Phi_f$ coincide.

The section of $\mu_{\mathcal{R}}$ for $\As(G)$ we shall use is constructed using a section $s_{S_G}$ for the $\mathcal{R}$ action of $\Kenh$ on all of $S_G$ that was described in \cite[Appendix B]{MT}, which in turn was based on the fact that for any matrix space $\End(M)$ (where $M$ is Hermitian), the map $B\mapsto \sqrt{-iB^*B}$ provides a section of the momentum map $\mu_{\mathcal{R}}$ of the $\mathcal{R}$ action of $\U(M)$ on $\End{M}$.  By \cite[Theorem 4.9]{Sj2} and (\ref{vinbdef}), we know that $\mu_{\mathcal{R}}(S_G)=\Kenh.Q_G$.  Hence we can write every element in $\mu_{\mathcal{R}}(S_G)$ as $k.(\mu,\mu+\sum_i m_i\alpha_i)$, following the notation of (\ref{thecone}).  

\begin{lemma}
The contraction map $\Phi_f$ from Corollary \ref{flowresult} can be described by \begin{equation}\label{Phif}\Phi_f:T^*K\cong G\subset S_G\rightarrow \As(G): g=\mathcal{L}_k s_{T^*K}(\chi)\mapsto \mathcal{L}_k s_{S_G}(\chi-\sum_i m_i\alpha_i)\end{equation} where $\chi\in \Kenh.Q_G$ can be written as $\chi=\tilde{k}.(\mu,\mu+\sum_i m_i\alpha_i)$ in the notation of (\ref{thecone}), for some $\tilde{k}\in \Kenh$.
\end{lemma}
This is the general form of (\ref{contrSL2}) for arbitrary semi-simple $G$.
\begin{proof}
It suffices to remark that the flow of $V_{\pi}$, and hence $\Phi_f$, is $(K\times K)$-equivariant, and preserves $s(\mu_{\mathcal{R}}(S_G))$ (since this is the case for all $V_{\det_i}$).   As a consequence we have that the flow of $V_{\pi}$ and $\Phi_f$ leave the moments maps for both $K$-actions on $S_G$ invariant.  Since the image of $\As(G)$ under the momentum map for the $\mathcal{R}$-action of $\Kenh$ lies in $K.\{ (\mu,\mu)\subset \mathfrak{t}_{\Kenh}^+ | \mu \in \Delta \}$, (\ref{Phif}) holds as it is the only map from $G$ to $\As(G)$ that satisfies these requirements.
\end{proof}

Crucially for us we have the following observation:
\begin{lemma}\label{samestab}
Given $\nu\in\mathfrak{k}^*$, the stabilizers for the $\mathcal{L}$-action of $K$  at $s_{(T^*K)^{\sc}}(\nu)$ in $(T^*K)^{\sc}$ or $s_{S_G}(\nu,\nu)\in\As(G)$ are identical.
\end{lemma}
\begin{proof}
The stabilizers for all orbits in $S_G$ of the $\Genh\times \Genh$ action are described in \cite[\S 0.7]{Vi1}.  In particular, orbits are enumerated by so-called \emph{essential} faces of the cone $Q_G$, indexed by (certain) pairs $(I,J)$, where both $I$ and $J$ are subsets of the sets of fundamental weights.  For $\As(G)$, we have that $I=\emptyset$, and any $J$ makes $(\emptyset,J)$ into an essential pair.  Specialising the description of \cite[Theorem 7]{Vi1} to this case, we see that the compact form of the stabilizer group of such an orbit (intersected with $K\times K$) is indeed the same as the stabilizer for an element of $s_{(T^*K)^{\sc}}(K.\Delta^{\vee}_J)$, namely $[K_J,K_J]\times [K_J,K_J]$. 
\end{proof}

With this in mind it is clear how we want to identify $(T^*K)^{\sc}$ with $\As(G)$ symplectically: equip the latter with the K\"ahler symplectic form coming from the matrix embedding (\ref{embedding}), and then simply put $$\widetilde{\Phi}: (T^*K)^{\sc}\rightarrow \As(G): \mathcal{L}_k \left(s_{(T^*K)^{\sc}}(\nu)\right)\mapsto\mathcal{L}_k\left(s_{S_G}(\nu,\nu)\right).$$

\begin{lemma} \label{2incarn}
This map $\widetilde{\Phi}$ is a $(K\times K)$-equivariant symplectomorphism.
\end{lemma}
\begin{proof} By Lemma \ref{samestab} $\widetilde{\Phi}$ is well-defined and a homeomorphism (in fact a diffeomorphism on each stratum).  It is also easily seen to be $(K\times K)$-equivariant: for the $\mathcal{L}$-action this follows by construction, and using the involution $\iota$ the same holds for $\mathcal{R}$.  To show it is a symplectomorphism, by continuity it suffices to show this on $\mathcal{L}_K s_{(T^*K)^{\sc}}(\Delta^{\vee})^o$.  It holds there since on the one hand this is symplectomorphic to the corresponding subset of $T^*K$ by  Proposition \ref{Xsymp}, and on the other hand the flow of gradient-Hamiltonian vector fields yields symplectomorphisms.
\end{proof}

\begin{corollary}
With the identification of Lemma \ref{2incarn} in mind, the contraction map of Corollary \ref{flowresult} is given by the map $\Phi$ from Section \ref{defcontr}.
\end{corollary}

\begin{remark}\label{noneedforGJS6}
As discussed above, symplectically the asymptotic semigroup $\As(G)$ of $G$ can be understood as $(T^*K)^{\sc}$.  In \cite[Section 6]{GJS} it is shown that  $E_{\mathcal{R}}(T^*K)$ can be equipped with a K\"ahler structure (and identified with $G\q U$) by embedding $G/U$ into a sum of irreducible representations $\bigoplus M_{\rho_i}$ of $G$ or $K$ (strictly speaking this is only shown for simply-connected, semi-simple $G$, but the results can be shown to hold for arbitrary linearly reductive $G$, see Appendix \ref{appendix}).   A similar result holds for $E_{\mathcal{L}}(T^*K)$.  The $\mathbb{T}$ that we factor out here acts on each $M_{\rho_i}$ through a single character.  It is now tempting to embed $(T^*K)^{\sc}=(E_{\mathcal{L}}(T^*K)\times E_{\mathcal{L}}(T^*K))\q_{\! 0} \mathbb{T}$ into $\bigoplus_i M_{\rho_i}\otimes M_{\rho_i}^*$ by mapping $((v_i)_i, (w_j)_j)$, with $v_i \in M_{\rho_i}$ and $w_j\in M_{\rho_j}^*$, to $(v_i\otimes w_i)_i$.  This is not what happens though, since for a Hermitian vector space $M$ the map $(M\times M^*)\q \U(1)\rightarrow M\otimes M^*: [(v,w)]\mapsto v\otimes w$ is not symplectic when using the linear K\"ahler structure on $M\times M^*$ as in (\ref{formEnd}).
\end{remark}

\section{Transfer}
Given a $G$-variety $X$ (which we shall assume to be semi-projective, such that the $G$-action is linearized), and a one-parameter sub-group, one can take the horospherical contraction of $X$, as discussed in Section \ref{horcontr}.  We begin by observing that by combining (\ref{generalhorcontr}) and Proposition \ref{basechange} this horospherical contraction is induced from the degeneration $(X\times S_G)\q G\rightarrow \A_G$ by base change:
\begin{center}
\begin{tikzcd}
 \A^1\times_{\A_G} \left(X\times S_G\right)\q G\ar{d}\ar{r} & \left(X\times S_G\right)\q G \ar{d} \\ \A^1 \ar{r} & \A_G
\end{tikzcd}
\end{center}
Note now that the gradient-Hamiltonian vector field $V_{\pi}$ associated with $\left(X\times S_G\right)\q G \times_{\A_G} \A^1\rightarrow \mathbb{C}$ is just the descent of $V_{\pi}$ on $\mathbb{C}\times_{\A_G} S_G$ restricted to the level-set of the momentum map under the quotient by $K$.

With this we can state our main result of this section:
\begin{theorem}\label{lastref}
Given any semi-projective $G$-variety $X$, which we equip with a K\"ahler structure through a $G$-linearized embedding $X\hookrightarrow \A^N\times \mathbb{P}^M$ for some $N$ and $M$, we have that the flow at time $1$ of the gradient-Hamiltonian vector field $V_{\pi}$ given by the horospherical contraction of $X$ extends to give the contraction map $\Phi_X$.
\end{theorem}
\begin{proof}
Begin by noting that we can symplectically embed $X$ into $X\times T^*K$ (and hence, by identifying $T^*K\cong G\cong \pi_G^{-1}(1)$ into $S_G$ or $\A^1\times_{\A_G}S_G$) by mapping $x\mapsto (x, s(\mu(m)))$, where $\mu$ is the momentum map for the action of $K$ on $X$, and as before $s$ is the preferred section for the momentum map for the $\mathcal{R}$-action of $K$ on $T^*K$.  It suffices now to remark that since the flow of $V_{\pi}$ on $\A^1\times_{\A_G}S_G$ (and hence on the non-singular locus of $\A^1\times_{\A_G}\left(X\times S_G\right)$) is invariant under the action of $K$, the induced contraction map $\Id\times \Phi:X\times G\rightarrow X\times \As(G)$ descends and indeed gives the contraction map for $V_{\pi}$.  By construction, cfr. (\ref{PhiX}), this coincides with $\Phi_X$. 
\end{proof}

\section{Application to branching problems}\label{branching}

Let $K$ and $L$ be connected, compact Lie groups with complex groups $K_{\C} = G$ and $L_{\C} = H$.  For any Lie group homomorphism $\phi: L \to K$, one can regard a complex $K$ representation $M$ as an $L$ representation by having $L$ act through $\phi$; this construction defines a pullback functor $\phi^*: Rep(K) \to Rep(L)$ on categories of finite-dimensional complex representations.  Recall that the categories $Rep(K)$ and $Rep(L)$ are semi-simple so the functor $\phi^*$ is determined by its values on the irreducible representations  $M_{\lambda}$ for $\lambda \in \mathfrak{X}_K^+$:

\begin{equation}
\phi^*(M_{\lambda}) = \bigoplus_{\eta \in \mathfrak{X}^+_L} \Hom_L(M_{\eta}, \phi^*(M_{\lambda}))\otimes M_{\eta}.\\
\end{equation}

\noindent
The branching problem associated to $\phi$ is the question of determining the dimensions of the multiplicity spaces $\Hom_L(M_{\eta}, \phi^*(M_{\lambda}))$ for all $\eta \in \mathfrak{X}^+_L$ and $\lambda \in \mathfrak{X}^+_K$. A coarser problem is to determine the set $P(\phi) \subset \big(\mathfrak{X}_L^+\times \mathfrak{X}_K^+\big)$ of dominant weights such that $\Hom_H(M_{\eta}, \phi^*(M_{\lambda})) \neq 0.$ 

\begin{example}[The Clebsch-Gordan rule]
The Clebsch-Gordan rule is perhaps the most elementary branching rule for a non-commutative Lie group, as it determines the branching law associated to the diagonal inclusion map $\delta: \SU(2) \to \SU(2) \times \SU(2)$.  Recall that the finite dimensional irreducible complex representations of $\SU(2)$ are indexed by positive integers, where $M_n, n \in Z_{\geq 0}$ is the vector space $\Sym^n(\C^2)$ with the symmetric action induced from $\C^2.$  The Clebsch-Gordan rule states that that $M_j\otimes M_k$ is multiplicity free and the multiplicity space 
$\Hom_{\SU(2)}(M_i, \delta^*(M_j \otimes M_k))$ is non-trivial 
 
 when $i + j + k \in 2\Z$ and $|i-j| \leq k \leq i + j$.  This latter condition can be recognized as the condition that $i, j, k$ must be the side-lengths of a triangle.  The set of $i, j, k$ satisfying these two properties therefore constitutes $P(\delta)$.
\end{example}

\begin{example}[The Pieri rule]
The fact that diagonal branching is  
multiplicity free significantly simplifies the $\SU(2)$ diagonal branching problem.  This feature also holds for the branching rule associated to the upper left diagonal inclusion $i_{n-1}:\U(n-1) \to \U(n)$, where the associated branching law is known as the Pieri rule (see \cite[Exercise 6.12]{FH}).  Recall that finite dimensional irreducible representations $M_{\lambda}$ of a unitary group $\U(k)$ are in bijection with weakly decreasing $k$-tuples of integers $\lambda$: $\lambda_1 \geq \ldots \geq \lambda_k.$  The Pieri rule states that 
 $\Hom_{\U(n-1)}(M_{\eta}, i_{n-1}^*(M_{\lambda}))$ is multiplicity free, and has dimension $1$ precisely when the entries of $\eta$ and $\lambda$ interlace: $\lambda_1 \geq \eta_1 \geq \lambda_2 \geq \ldots \geq \lambda_{n-1}\geq \eta_{n-1} \geq \lambda_n$.  The entries of $\eta$ and $\lambda$ then fit into an interlacing pattern, see Figure \ref{interlace}. The set $P(i_{n-1})$ can then be identified with the set of all interlacing patterns with top row length $n$.
\begin{figure}[H]
\begin{tikzpicture}
\draw [->] (-1,1.2) -- (-1.8, 0);
\draw [->] (-2, 0) -- (-2.8, 1.2);
\draw [->] (-3,1.2) -- (-3.8, 0);
\draw [->] (-4, 0) -- (-4.8, 1.2);
\draw [->] (-5,1.2) -- (-5.8, 0);
\draw [->] (-6, 0) -- (-6.8, 1.2);
\draw [->] (-7,1.2) -- (-7.8, 0);
\draw [->] (-8, 0) -- (-8.8, 1.2);
\node at (-1.9, -.2) {$\eta_4$};
\node at (-3.9, -.2) {$\eta_3$};
\node at (-5.9, -.2) {$\eta_2$};
\node at (-7.9, -.2) {$\eta_1$};
\node at (-.9, 1.4) {$\lambda_5$};
\node at (-2.9, 1.4) {$\lambda_4$};
\node at (-4.9, 1.4) {$\lambda_3$};
\node at (-6.9, 1.4) {$\lambda_2$};
\node at (-8.9, 1.4) {$\lambda_1$};
\end{tikzpicture}
\caption{An interlacing diagram.}
\label{interlace}
\end{figure}
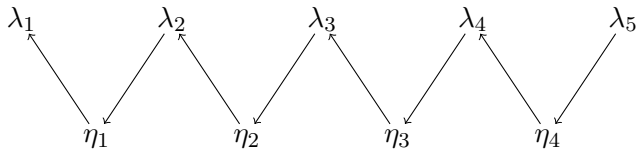
\end{example}

In this section we review how the multiplicity spaces $\Hom_L(M_{\eta}, \phi^*(M_{\lambda}))$ can be studied using algebraic varieties called branching varieties, and we apply our results on the Vinberg monoid to these spaces.  We direct the reader to the work of Howe, Tan, and Willenbring \cite{HTW} on the combinatorial commutative algebra of branching varieties (also see \cite{M14} and \cite{M}). The symplectic geometry of branching varieties has been studied  by Berenstein and Sjamaar, \cite{BSj}.

\subsection{Affine branching varieties}

As above, we fix $L, K$ to be compact connected Lie groups with associated complex groups $H, G$, respectively.  Let $\phi: L \to K$ be a map of compact Lie groups which lifts to a map on the associated complex groups, also denoted $\phi$.  Recall that the quotient $H \q U_H \times G \q U_G$ carries an algebraic  $H\times T_H \times G \times T_G$ action, and its coordinate ring has the isotypical  decomposition: $\C[H\q U_H \times G\q U_G] = \bigoplus_{(\eta,\lambda) \in \mathfrak{X}^+_H \times \mathfrak{X}^+_G} M_{\eta} \otimes M_{\lambda}$.  

\begin{definition}
The affine branching variety $X(\phi)$ is the left diagonal GIT-quotient of the product  $H\q U_H \times G\q U_G$ by the group $H$ taken with respect to the left action on $H \q U_H$ and the left action through $\phi$ on $G \q U_G$:

\begin{equation}
X(\phi ) = H \qlb \big( H\q U_H \times G\q U_G \big) .
\end{equation}

\end{definition}

The variety $X(\phi)$ has a residual right action by the torus $T_H \times T_G$, and the corresponding isotypical components of $\C[X(\phi)]$ are the invariant spaces $[M_{\eta}\otimes M_{\lambda}]^H$ (from now on we drop the $\phi^*$ where no confusion will result).  Using the natural isomorphism $[M_{\eta}\otimes M_{\lambda}]^H \cong \Hom_H(M_{\eta}^*, M_{\lambda})$, we can rewrite this decomposition:

\begin{equation}
\C[X(\phi)] = \bigoplus_{\eta \in \mathfrak{X}^+_H, \lambda \in \mathfrak{X}^+_G} \Hom_H(M_{\eta}, M_{\lambda}).\\
\end{equation}

\noindent
In particular, the set of $T_H \times T_G$ characters supported in this decomposition coincides with $P(\phi)$, which allows us to draw several conclusions about $P(\phi)$. First, we see that as the graded support of a graded algebra, $P(\phi)$ is a submonoid of $\mathfrak{X}^+_H\times \mathfrak{X}^+_G$.   We let $\Delta(\phi)$ be the convex hull of $P(\phi)$ -- this is a cone in the product $\Delta_H \times \Delta_G$ which comes equipped with two projections $\pi_G: P(\phi) \to \Delta_G$, $\pi_H: P(\phi) \to \Delta_H$.    Both $H \q U_H$ and $G \q U_G$ are affine varieties of finite type and $H$ is reductive, so it follows that $X(\phi)$ is of finite type and $\C[X(\phi)]$ is finitely generated.  As a consequence, $P(\phi)$ is a finitely generated monoid and $\Delta(\phi)$ is a polyhedral cone.  Furthermore, any rational point $(\eta, \lambda) \in \Delta(\phi)$ (with respect to the $H\times G$ weight lattice) can be scaled to give a member of $P(\phi)$, so $\Hom_H(M_{k\eta}, M_{k\lambda}) \neq 0$ for sufficiently large integers $k$.

\subsection{Projective branching varieties}

Let $P \subset H$ and $Q \subset G$ be parabolic subgroups containing $T_H, T_G$ and $U_H, U_G$ respectively.  The quotients $H/P$ and $G/Q$ are projective flag varieties of $H$ and $G$.  A dominant weight $\eta \in \mathfrak{X}^+_H$ corresponds to a character $\chi_{\eta}: T_H \to \C^*$; if this character lifts to a character of $P$, it defines a corresponding line bundle $\mathbf{L}_{\eta}$ on $H/P$.  Recall that by the Borel-Bott-Weil theorem, the global sections of this line bundle can be identified with $M_{\eta}$ as an $H$-representation.   

\begin{definition}
The projective branching variety $X_{\eta, \lambda}(\phi)$ with line bundle $\mathbf{L}_{\eta, \lambda}$ is defined to be the GIT-quotient of $H/P\times G/Q$ with respect to the left diagonal $H$-action and the $H$-action on the line bundle $\mathbf{L}_{\eta} \boxtimes \mathbf{L}_{\lambda}.$ 
\end{definition}

The character $\chi_{\eta}$ defines a linearization of the action of $T_H$ on $H \q U_H$ equipped with  the trivial line bundle.  In this way the pair $(H/P, \mathbf{L}_{\eta})$ can also be recovered by way of a projective GIT-quotient of the affine variety $H \q U_H$ by $T_H$.    

\begin{proposition}
The projective variety $X_{\eta, \lambda}(\phi)$ is the GIT-quotient of the affine branching variety $X(\phi)$ with respect to the $T_H \times T_G$-linearization of the trivial line bundle of $X(\phi)$ given by the product character $\chi_{\eta, \lambda}$ of $\chi_{\eta}$ and $\chi_{\lambda}$. 
\end{proposition}

\subsection{Branching varieties as reductions}

Now we replace $H\q U_H, G\q U_G$ with their symplectic analogues $E_{\mathcal{R}}(T^*L), E_{\mathcal{R}}(T^*K)$, and realize branching varieties as symplectic reductions.  We give the imploded cotangent bundle its natural symplectic structure coming from the cotangent bundle, and we give the flag varieties $H/P$ and $G/Q$ their Kostant-Kirillov symplectic forms as the coadjoint orbits $\mathcal{O}_{-\eta}, \mathcal{O}_{-\lambda}$.  

\begin {proposition}\label{affine-branching-reduction}
The affine branching variety $X(\phi)$ can be identified with the left diagonal reduction of $E_{\mathcal{R}}(T^*L) \times E_{\mathcal{R}}(T^*K)$ by $L$ at level $0$: 

\begin{equation}
X(\phi) = L _{\ 0}{\!\qlb}{} \big(E_{\mathcal{R}}(T^*L) \times E_{\mathcal{R}}(T^*K)\big).
\end{equation}

For any $(\eta, \lambda) \in \mathfrak{X}^+_H\times \mathfrak{X}^+_G,$ the projective branching variety $X_{\eta, \lambda}(\phi)$ can be identified with the simultaneous reduction of $E_{\mathcal{R}}(T^*L) \times E_{\mathcal{R}}(T^*K)$ with respect to the left diagonal action of $L$ at level $0$ and the right action by $\mathbb{T}_L\times \mathbb{T}_K$ at level $(\eta, \lambda) \in \Delta_L\times \Delta_K$:  

\begin{equation}
X_{\eta, \lambda}(\phi) = L _{\ 0}{\!\qlb}{} \big(E_{\mathcal{R}}(T^*L) \times E_{\mathcal{R}}(T^*K)\big) \qb_{\!\!(\eta, \lambda)} \mathbb{T}_L \times \mathbb{T}_K.\\
\end{equation}

\end{proposition}

\begin{proof}
In both cases we begin with a Hamiltonian $(L \times K \times \mathbb{T}_L \times \mathbb{T}_K)$-structure on an algebraic subvariety of a complex vector space, equipped with an invariant Hermitian form, so this Proposition follows from the results of Sjamaar-Lerman \cite{SjL} and Kempf-Ness \cite{KN}. 
\end{proof}

The action of $L$ on $E_{\mathcal{R}}(T^*K)$ has momentum map $\mu_{L}(k, w) = d\phi^*(k w),$ where $d\phi^*: \mathfrak{k}^* \to \mathfrak{l}^*$ is the dual of the map on Lie algebras induced by $\phi.$  The space $X(\phi)$ is then the $L$-quotient of the subspace $\big\{ \big((p, u),(k, w)\big)\, \big|\, pu + d\phi^*(kw) = 0\big\}$.  Let $\mu_{\mathbb{T}_L \times \mathbb{T}_K}$ be the momentum map of the residual action of $\mathbb{T}_L \times \mathbb{T}_K$ on $X(\phi)$.  For the following proposition see also \cite[Section 3]{BSj}. 

\begin{proposition}\label{branch-cone-image}
The momentum image of $X(\phi)$ with respect to its residual action by $\mathbb{T}_L\times \mathbb{T}_K$ is $\Delta(\phi).$
\end{proposition}

\begin{proof}
If $(\eta, \lambda) \in (\mathfrak{X}^+_H \times \mathfrak{X}^+_G) \cap \mu_{\mathbb{T}_L\times \mathbb{T}_K}(X(\phi))$ then the GIT-quotient $H\! \tensor[_{(\eta, \lambda)}]{\!\qlb}{} \big(H/P \times G/Q\big)$ is non-empty. It follows that $\Hom_H(M_{k\eta}, M_{k\lambda}) \neq 0$ for some sufficiently large integer $k$, so that $(\eta, \lambda) \in \Delta(\phi)$.  If $(\eta, \lambda) \in P(\phi)$, we can run this argument in reverse to conclude that the reduction $$X_{(\eta, \lambda)}(\phi) =  X(\phi) \q_{(\eta, \lambda)} \mathbb{T}_L\times \mathbb{T}_K= L _{\ 0}{\!\ql}{} \big(\mathcal{O}_{-\eta}\times \mathcal{O}_{-\lambda}\big)$$ must be non-empty, so that $(\eta, \lambda)$ is in the momentum image.  It follows that  $(\mathfrak{X}^+_H \times \mathfrak{X}^+_G) \cap \mu_{\mathbb{T}_L\times \mathbb{T}_K}(X(\phi)) = P(\phi)$.  Since any rational point in $ \mu_{\mathbb{T}_L\times \mathbb{T}_K}(X(\phi))$ can be scaled by some integer to obtain a point in $\mathfrak{X}^+_H \times \mathfrak{X}^+_G$, it follows that the rational points of $ \mu_{\mathbb{T}_L\times \mathbb{T}_K}(X(\phi))$ and $\Delta(\phi)$ coincide, so these cones must likewise coincide. 
\end{proof}

\begin{example}[Diagonal $\SU_2(\C)$ branching]
For what follows see \cite{HaKn} and \cite{HMM}. 

Let $\delta_n: \SU(2) \to \SU(2)^{n-1}$ (respectively $\delta_n: \SL_2(\C) \to \SL_2(\C)^{n-1}$) be the diagonal embedding.  The branching variety $X(\delta_n)$ is the diagonal GIT-quotient of $(\SL_2(\C) \q U)^n$ by $\SL_2(\C)$, and the diagonal reduction of $E_{\mathcal{R}}(\SU(2))^n$ by the left action of $\SU(2)$ at level $0$.  The space $\SL_2(\C) \q U$ is isomorphic to $\C^2$ both as a variety and a symplectic manifold, so $X(\delta_n)$ is isomorphic to $\SL_2(\C) \ql \C^{2\times n}$.  By the first fundamental theorem of invariant theory (see \cite{Dol}), $X(\delta_n)$ is the affine cone of the Grassmannian variety $Gr_2(\C^n)$ with respect to its Pl\"ucker embedding in the projective space $\mathbb{P}(\bigwedge^2(\C^n))$. The action by the maximal torus $\mathbb{T} \subset \SU(2)$ on $\C^2$ is the standard diagonal action, and the residual $\mathbb{T}^n$ action on $X(\delta_n)$ agrees with the action of the maximal torus of $\U(n)$ on the Grassmannian variety. The momentum image $\Delta(\delta_n)$  is the convex hull of the rays through points of the form $(0, \ldots, 1, \ldots, 1, \ldots, 0)$; this is a cone over the so-called second hypersimplex.  A point $\vec{r}$ is in $P(\delta_n)$ if and only if its entries can be the sides of an $n$-sided polygon. 

The projective branching varieties $X_{\vec{r}}(\delta_n)$ are precisely the weight varieties of the Grassmannian, namely they are GIT-quotients of $\Gr_2(\C^n)$ by the maximal torus of $\GL_n(\C).$  As explained in \cite{HMM}, the symplectic geometry of these spaces fits naturally with their interpretation as moduli spaces of polygons in $\R^3.$  In particular, the reduction $X_{\vec{r}}(\delta_n)$ is the moduli space of $n$-sided polygons in $\R^3$ with prescribed sidelengths $r_1, \cdots, r_n$, up to isometry.  Likewise $X(\delta_n)$ can be viewed as the moduli space of $n$-sided polygons in $\R^3$ with an $\SU(2)$-framing oriented along each edge.  In this latter context, the residual $\mathbb{T}^n$-action be interpreted as the action which spins each of these $n$ frames around its associated edge.

\end{example}

\subsection{Branching degeneration}\label{secondrefsection}

We recall the theory of branching degenerations, cfr. \cite{M}.  These are flat degenerations of branching varieties
associated to factorizations of the map $\phi$ in the category of connected linearly reductive groups:

$$
\begin{CD}
H @>\pi>> F @>\psi>> G.\\
\end{CD}
$$ 

\noindent
In this case, the pullback functor $\phi^*$ is the composition $\psi^* \circ \pi^*$, and any multiplicity space $\Hom_H(M_{\eta}, M_{\lambda})$ may be written as a direct sum:

\begin{equation}
\Hom_H(M_{\eta}, M_{\lambda}) = \bigoplus_{\beta \in \mathfrak{X}^+_F} \Hom_H(M_{\eta}, M_{\beta})\otimes \Hom_F(M_{\beta}, M_{\lambda}).\\
\end{equation}

\noindent
Naively, this implies that as a vector space $\C[X(\phi)]$ can be realized as the subspace of $$\C[X(\psi) \times X(\pi)] = \bigoplus_{\eta \in \mathfrak{X}^+_H, \beta_1, \beta_2 \in \mathfrak{X}^+_F, \lambda \in \mathfrak{X}^+_G}  \Hom_H(M_{\eta}, M_{\beta_1})\otimes \Hom_F(M_{\beta_2}, M_{\lambda})$$ where $\beta_1 = \beta_2$.  Notice that this subspace is the invariant subalgebra $$\C[X(\psi) \times X(\pi)] ^{T_F} \subset \C[X(\psi) \times X(\pi)]$$ by an action of the antidiagonal subtorus $T_F \subset T_F \times T_F$.  We let $X(\psi, \pi)$ be the affine GIT-quotient of $\C[X(\psi) \times X(\pi)]$ by this action.

The map $\psi$ can be used to write the $G$-variety $G\q U_G$ as the $F$-variety $\big(F \times G\q U_G\big)\q F$.  The factorization $\phi = \psi \circ \pi$ then implies that the action of $H$ on $G\q U_G = \big(F \times G\q U_G\big)\q F$ factors through the residual left $F$-action. We get the following inefficient but useful description of $X(\phi)$ as a consequence:

\begin{equation}
X(\phi) = H \ql \big(H\q U_H \times F \times G \q U_G\big)\q F,
\end{equation}

\begin{equation}
\C[X(\phi)] = \C[H\q U_H \times F \times G\q U_G]^{H \times F} \subset \C[H\q U_H \times F \times G\q U_H].
\end{equation}

Recall (see Theorem \ref{ahcontract}) that  the $F\times F$-stable valuations on the coordinate ring $\C[F]$ have the structure of a convex cone identified with $\Delta_F^{\vee}$.  

\begin{proposition}\label{branching-valuations}
The following hold for $H\q U_H \times F \times G \q U_G$ and its quotient $X(\phi).$
   
\begin{enumerate}
\item\label{itemone} For any triple $(f, h, g) \in \Delta_H^{\vee}\times \Delta_F^{\vee}\times \Delta_G^{\vee}$ there is a valuation on $\C[H\q U_H \times F \times G \q U_G]$ and $\C[X(\phi)].$\\
\item\label{itemtwo} With $(f, h, g)$ as above, if $h$ is a member of the top face of $\Delta_F^{\vee}$, then the associated graded algebra of the corresponding valuation on $\C[H\q U_H \times F \times G \q U_G]$  is the coordinate ring $\C[H\q U_H \times \As(F) \times G \q U_G]$.\\ 
\item \label{itemthree} With $(f, h, g)$ as above, if $h$ is a member of the top face of $\Delta_F^{\vee}$, then the associated graded algebra of the corresponding valuation on $\C[X(\phi)]$  is the coordinate ring $\C[H\q U_H \times \As(F) \times G \q U_G]^{H \times F}$.\\
\end{enumerate}
\end{proposition}

\begin{proof}
Using Theorem \ref{ahcontract}, it is straightforward to see that the filtration of $H\q U_H \times F \times G \q U_G  $ defined by the spaces $$\mathcal{F}^{f, h, g}_{\leq m} = \bigoplus_{\eta(f) + \beta(h) + \lambda(g) \leq m} M_{\eta}\otimes \Hom(M_{\beta}, M_{\beta}) \otimes M_{\lambda}$$ has associated graded algebra $\C[H\q U_H \times \As(F) \times G \q U_G]$, which proves (\ref{itemtwo}).  Everything in sight is $H\times F$-linear, so (\ref{itemtwo}) implies (\ref{itemthree}).  Furthermore, all of the associated graded algebras we have just encountered are domains, which implies that each of these filtrations comes from a valuation and proves item (\ref{itemone}).
\end{proof}

If $h$ is chosen in the top face of $\Delta_F^{\vee}$, the associated graded algebra of the filtration $\mathcal{F}^{f, h, g}$  on $\C[X(\phi)]$ is the following algebra of invariants:

\begin{equation}
\gr_{\mathcal{F}^{f, h, g}}(\C[X(\phi)]) = \C\big[H\q U_H \times \big(F \q U_- \times U_+ \ql F\big)\q T_F \times G\q U_G\big]^{H \times F} = \C[X(\psi, \pi)].\\
\end{equation}

\noindent
Here we have used the fact that $U_+ \ql F$ can be isomorphically identified with $F\q U_-$ as $F \times T_F$-varieties.  From now on we let $C(\psi, \pi)$ denote the cone $\Delta_H^{\vee}\times \Delta_F^{\vee}\times \Delta_G^{\vee}$, thought of as a cone of valuations on $\C[X(\phi)].$

The construction of $X(\phi)$ as a $H\times F$ quotient of $H \q U_G \times F \times G \q U_G$  can be carried out by replacing every instance of $F$ above with the Vinberg monoid $S_F$.  By doing so, we produce a flat family $$E(\psi, \pi) = H \qlb \big(H \q U_H \times S_F \times G \q U_G\big) \qb F$$ over $\mathbb{A}_F$ associated to the factorization $\phi = \psi \circ \pi$.  By Theorem \ref{ahcontract} the fiber of $E(\psi, \pi)$ over $0 \subset \mathbb{A}_F$ is $X(\psi, \pi)$.  For any valuation $h \in \Delta_F^{\vee}$ there is a $1$-parameter family in $\mathbb{A}_F$, and base-changing with respect to this family produces the degeneration corresponding to the filtration $\mathcal{F}^{0, h, 0}$ described above.

\begin{center}
\begin{tikzcd}
\A^1\times_{\A_F} E(\psi, \pi) \ar{r} \ar{d}{\pi_h} & E(\psi, \pi)\ar{d}{\pi_F}
 \\ \A^1\ar{r} & \A_F .
\end{tikzcd}
\end{center}

 Notice that $X(\psi, \pi)$ has a residual $T_F$ action, accordingly the coordinate ring $\C[X(\psi, \pi)]$ has the following decomposition into $T_H \times T_F \times T_G$ isotypical components:

\begin{equation}\label{iso-decomp-degen-branch}
\C[X(\psi, \pi)] = \bigoplus_{\eta \in \mathfrak{X}^+_H, \beta \in \mathfrak{X}^+_F, \lambda \in \mathfrak{X}^+_G} \Hom_H(M_{\eta}, M_{\beta}) \otimes \Hom_F(M_{\beta}, M_{\lambda}).\\
\end{equation}

\noindent
As a consequence, the support of the $T_H \times T_F \times T_G$ isotypical decomposition  of $\C[X(\psi, \pi)]$ is the fiber product monoid $$P(\psi, \pi)= P(\pi) \times_{\mathfrak{X}^+_F} P(\psi) = \{(w,v) \in P(\pi)\times P(\psi)| \pi_F(w) = \pi_F(v)\} \subset P(\pi) \times P(\psi).$$  This monoid is polyhedral and comes with a canonical projection $$p_{\psi, \pi}: P(\psi, \pi) \to P(\psi \circ \pi) = P(\phi) \subset \Delta_H \times \Delta_G.$$  
We let $\Delta(\psi, \pi)$ be the convex hull of $P(\psi, \pi)$; this is by definition the fiber product $\Delta(\psi)\times_{\Delta_F}\Delta(\pi)$.

For any pair of weights $(\eta, \lambda) \in P(\phi)$ we can form a new flat family over $\mathbb{A}_F$ by taking a GIT-quotient by $T_H \times T_G$ with respect to the character $\chi_{\eta, \lambda}$:

\begin{equation}
E_{\eta, \lambda}(\psi, \pi) = E(\psi, \pi) \q_{\! (\eta, \lambda)} T_H \times T_G,
\end{equation}

\noindent
this space is obtained by taking $\Proj$ of the following $\Z_{\geq 0}$-graded subalgebra of $\C[E(\psi, \pi)]$
(see Equation \ref{iso-decomp-degen-branch}):

\begin{equation}
\C[E_{\eta, \lambda}(\psi, \pi)] = \bigoplus_{k \in \Z_{\geq 0}, \beta - \alpha \in \mathfrak{W}} \Hom_H(M_{k\eta}, M_{\alpha})\otimes \Hom_F(M_{\alpha}, M_{k\lambda}) t^{\beta}.
\end{equation}

\noindent 
(Recall that $\mathfrak{W}$ is the root lattice).  The next proposition is a straightforward calculation. 

\begin{proposition}\label{lastrefbis}
The fiber over $0 \in \mathbb{A}_F$ in the family $E_{\eta, \lambda}(\psi, \pi)$ is the GIT-quotient of $X(\psi)\times X(\pi)$ by $T_H \times T_F \times T_G$, taken with respect to the linearization of the trivial line bundle on $X(\psi)\times X(\pi)$ defined by the product character $\chi_{(\eta, 0, \lambda)}.$   A general fiber of this family is isomorphic to $X_{\eta, \lambda}(\phi).$
\end{proposition}

\noindent
From now on we let  $X_{\eta, \lambda}(\psi, \pi)$ denote the GIT-quotient $[X(\psi)\times X(\pi)] \q_{(\eta, 0, \lambda)} T_H \times T_F \times T_G$.

\subsection{Chains of subgroups}

The constructions of the cone of valuations $C(\psi, \pi)$ on the coordinate ring $\C[X(\phi)]$, and the degeneration $X(\phi) \Rightarrow [X(\psi, \pi)]$ from a factorization $\phi = \pi \circ \psi$ can be readily generalized to a chain of maps:
$$
\begin{CD}
G_0 @>\phi_1>> \ldots @>\phi_n>> G_n.\\
\end{CD}
$$

\noindent
We leave the details of the following to the reader (the result follows from repeatedly applying Propositions \ref{branching-valuations} and \ref{lastrefbis}):

\begin{proposition}\label{chainbranch}
For $1 \leq i \leq n$ let $\phi_i: G_{i-1} \to G_i$ be a map of connected, linearly reductive groups, and let $X(\phi_n \circ \ldots \circ \phi_1)$ be the branching variety of the composition.  Let $C(\vec{\phi}) = \prod_{i = 0}^n \Delta_i^{\vee}$ and $\eta \in \mathfrak{X}^+_{G_0}, \lambda \in 
\mathfrak{X}^+_{G_n}.$

\begin{enumerate}
\item Each point $\vec{h} \in C(\vec{\phi})$ defines a valuation on $\C[X(\phi_n \circ \ldots \circ \phi_1)].$\\
\item  The associated graded algebra of the valuation associated to an integral point in the interior is $\C\big[(X(\phi_n) \times \ldots \times X(\phi_1))\q\prod T_i\big] = \C[X(\vec{\phi})],$ where the $T_i \subset T_i \times T_i$ acts antidiagonally on $X(\phi_i) \times X(\phi_{i-1})$.\\
\item There is a flat family $E(\vec{\phi}) = [G_0 \q U_0 \times \prod_{i =1}^{n-1} S_{G_i} \times G_n \q U_n] \q \prod_{i = 1}^{n-1} G_i$ over $\prod_{i = 1}^{n-1} \mathbb{A}_i$ with general fiber $X(\phi_n \circ \ldots \circ \phi_1)$ and $0$ fiber $X(\vec{\phi})$.  For integral $\vec{h} \in C(\vec{\phi})$ in the interior, the associated family obtained by base-change has the same special fiber $X(\vec{\phi})$.\\
\item The support of the isotypical decomposition of the coordinate ring of the special fiber under the action of $\prod_{i = 1}^{n-1} T_i$ is the fiber product monoid $P(\vec{\phi}) = P(\phi_1)\times_{\mathfrak{X}^+_{G_1}}\ldots \times_{\mathfrak{X}^+_{G_{n-1}}} P(\phi_n).$ This monoid comes with a projection $p_{\vec{\phi}}: P(\vec{\phi}) \to P(\phi_n \circ \ldots \circ \phi_1).$\\
\item The cone $C(\vec{\phi})$ and the family $E(\vec{\phi})$ induce a cone of valuations and a degeneration $E_{\eta, \lambda}(\vec{\phi})$ of $X_{\eta, \lambda}(\phi_n \circ \ldots \circ \phi_1)$. This degeneration has $0$ fiber $X_{\eta, \lambda}(\vec{\phi})$.  In particular the family $E_{\eta, \lambda}(\vec{\phi})$ is computed by taking a $T_0\times T_n$-GIT-quotient of $E(\vec{\phi})$ with respect to the character defined by $(\eta, \lambda).$ \\
\end{enumerate}

\end{proposition}

\begin{example}[Tree factorizations and diagonal branching]

The diagonal map $\delta_n: \SL_2(\C) \to \SL_2(\C)^{n-1}$ has distinguished class of factorizations which are combinatorially indexed by trees $\tree$ with $n$ leaves labeled from the set $[n].$  This construction is described in full detail in \cite{M}, we will outline it here. 

Fix a tree $\tree$ with vertex set $V(\tree)$ and edge set $E(\tree).$  The leaves $\ell \in [n]$ are each identified with a factor of the product $\SL_2(\C) \q U_+ \times (\SL_2(\C) \q U_+)^{n-1}$ used in the definition of $X(\delta_n).$   We place an orientation on each edge $e \in E(\tree)$ in the unique way so that $1 \in [n]$ is the unique source and the other leaves are the only sinks.  Let $S_i$ be the set of those vertices of distance $i$ from the leaf $1 \in [n]$, and let $L_i$ be the set of leaves in $\cup_{j =1}^i S_i$.    Let $G_i$ be the product $\SL_2(\C)^{S_i}\times \SL_2(\C)^{L_i}$.  The orientations on $E(\tree)$ can be used to define a chain of maps $\phi_i: G_i \to G_{i+1}$ as follows. Each component $\SL_2(\C) \subset \SL_2(\C)^{L_i}$ is mapped by the identity to its corresponding copy in $\SL_2(\C)^{L_{i+1}}$, and each component $\SL_2(\C) \subset \SL_2(\C)^{S_i}$ is mapped diagonally into the product of components $\SL_2(\C) \subset \SL_2(\C)^{S_{i+1}}\times \SL_2(\C)^{L_{i+1}}$ whose corresponding vertices are connected to the vertex in question.  The maps $\phi_i$ define a factorization of $\delta_n$, so we may apply Proposition \ref{chainbranch}. 

In the case that the tree $\tree$ is trivalent, the cone $C_{\tree}$ of associated valuations has dimension $2n-3$.  If a tree $\tree'$ can be obtained from a tree $\tree$ by contracting a set of non-leaf edges $S$,  there is an inclusion $C_{\tree'} \subset C_{\tree}$ defined by considering those $v \in C_{\tree}$ with $v(e) = 0$ for $e \in S$.  Using these inclusions one can construct a polyhedral complex $T(n)$ which a tropical geometer may recognize as a the space of phylogenetic trees studied by Speyer and Sturmfels in \cite{SpSt}. The complex $T(n)$ can be mapped to the tropical Grassmannian variety constructed in \cite{SpSt} by sending each $v \in T(n)$ to the vector $(v(p_{1,2}), \ldots, v(p_{n-1,n}))$, where the $p_{ij}$ are the Pl\"ucker generators of $\C[X(\delta_n)]$.

The degenerations of $X(\delta_n)$ corresponding to a valuation $v \in C_{\tree} \subset T_n$ are described in \cite{HMM} (see also \cite{M14} for the case of a general reductive $G$).  For a non-zero weighting of the edges of a trivalent tree $\tree$, the degeneration is the affine toric variety $X(\tree)$  associated to the affine branching semigroup $P_{\tree}$.  Here $P_{\tree}$ is the set of integer weightings $w: E(\tree) \to \Z_{\geq 0}$ such that the three integers $w(i), w(j), w(k)$ assigned to edges sharing a common vertex have an even sum (i.e. $w(i) + w(j) + w(k) \in 2\Z$), and satisfy triangle inequalties: $|w(i) - w(k)| \leq w(j) \leq w(i) + w(k)$.

For each $\vec{r} \in P(\delta_n)$ one also obtains a degeneration of $X_{\vec{r}}(\delta_n)$ to the projective toric variety $X_{\vec{r}}(\tree)$ corresponding to a convex polytope $P_{\tree}(\vec{r})$.  The geometry of these toric varieties is the subject of \cite{HMM}.

\end{example}

\subsection{Branching contraction}

 Following the previous subsection, assume that the maps $\psi \circ \pi = \phi$
are induced from maps of compact semi-simple Lie groups:
$$
\begin{CD}
L @>\pi>> J @>\psi>> K.
\end{CD}
$$
Using the universal Hamiltonian property of the cotangent bundle $T^*J$ (see Section \ref{prelim}), we realize the branching variety as the following symplectic reduction: 
\begin{equation}
X(\phi) = L _{\ 0}{\!\qlb}{} \big(E_{\mathcal{R}}(T^*L) \times T^*J \times E_{\mathcal{R}}(T^*K)\big)\qb_{\! 0} J.
\end{equation}

\noindent
Likewise, the contracted branching variety $X(\psi, \pi)$ can be constructed with a symplectic reduction. 

\begin{proposition}
The contracted branching variety $X(\psi, \pi)$ can be identified with the symplectic reduction $$\big(X(\pi)\times X(\psi)\big)\q_{\! 0}\mathbb{T}_J =  L _{\ 0}{\!\qlb}{} \big(E_{\mathcal{R}}(T^*L) \times (T^*J)^{\sc} \times E_{\mathcal{R}}(T^*K)\big)\qb_{\! 0} J.$$  Furthermore, there is a residual Hamiltonian action of $\mathbb{T}_L \times \mathbb{T}_J \times \mathbb{T}_K$ on $X(\psi, \pi)$ with momentum image equal to $\Delta(\psi, \pi).$ 
\end{proposition}

\begin{proof}
We may apply the theorem of Kempf and Ness \cite{KN} to $E_{\mathcal{R}}(T^*L) \times (T^*J)^{\sc} \times E_{\mathcal{R}}(T^*K)$ as its symplectic structure comes from $L \times J$-stable embedding in a Hermitian vector space.  

There is a residual $\mathbb{T}_L$ action coming from the right action on $E_{\mathcal{R}}(T^*L)$, a residual $\mathbb{T}_K$ action coming from the right action on $E_{\mathcal{R}}(T^*K)$, and a residual $\mathbb{T}_J$ action coming from the action on $(T^*J)^{\sc}$.  By Proposition \ref{branch-cone-image} and the definition of $(T^*J)^{\sc}$, the momentum image of this action is the subset of $\Delta(\psi)\times \Delta(\pi) \subset \Delta_L \times \Delta_J \times \Delta_J \times \Delta_K$ where the $\Delta_J$ components coincide.  By definition, this is the cone $\Delta(\psi, \pi).$ 
\end{proof}

The symplectic contraction map $\Phi: T^*J \to T^*J^{\sc}$ clearly defines a surjective, continuous, and proper map $\widehat{\Phi}: E_{\mathcal{R}}(T^*L) \times (T^*J) \times E_{\mathcal{R}}(T^*K) \to E_{\mathcal{R}}(T^*L) \times (T^*J)^{\sc} \times E_{\mathcal{R}}(T^*K).$  As $\Phi$ is a map of Hamiltonian $(J\times J)$-spaces, $\widehat{\Phi}$ descends to give the branching contraction map:

\begin{equation}
\Phi_{\pi, \psi}: X(\phi) \to X(\psi, \pi).\\
\end{equation}

\noindent
One evaluates this map  on a class $(p, w),(k, v) \in X(\phi)$ by first computing $d\psi^*(kv) \in \mathfrak{j}^*,$ and then finding a $j \in J$ which diagonalizes this element: $jd\psi^*(kv) \in \Delta_J.$  The image $\Phi_{\psi, \pi}\big([(p, w)(k, v)]\big)$ is then the equivalence class $[(p, w),(j^{-1}, j d\psi^*(kv))]\times[(j, d\psi^*(k v),(k, v)] \in X(\psi, \pi)$.

\begin{proposition}
The branching contraction map $\Phi_{\pi, \psi}$ is  surjective, continuous, and proper. 
\end{proposition}

\begin{proof}
This is a consequence of Proposition \ref{Gcontract}. 
\end{proof}

We let $\mu_J: E_{\mathcal{R}}(K) \to \mathfrak{j}^*$ be the momentum map for the $\mathcal{L}$-action of $J$ used in the construction of the contracted branching variety.  The image of $\mu_J$ is $p_J(P(\psi))$, the $\R_{\geq 0}$ span of the set of weights $\beta$ which appear in the $J$-decomposition of some representation of $K.$  If the principal face in $\Delta_J$ of $E_{\mathcal{R}}(T^*K)$ under this map is the open chamber $\Delta_{J, o}$, then Proposition \ref{Gcontract} implies that the induced map $$\widetilde{\Phi}: \big(T^*J \times E_{\mathcal{R}}(T^*K)\big)\q_{\! 0} J \to \big((T^*J)^{\sc} \times E_{\mathcal{R}}(T^*K)\big)\q_{\! 0} J$$ is a symplectomorphism on a dense open subset $$\big(T^*J^o \times E_{\mathcal{R}}(T^*K)\big)\q_{\! 0} J \subset  \big(T^*J \times E_{\mathcal{R}}(T^*K)\big)\q_{\! 0} J.$$  Proposition \ref{Xsymp} guarantees that this remains the case if the principal face $\Delta_{J, I} \subset \Delta_J$ is not open. In both cases, the image of
$\mu_J$ coincides with the momentum image of the residual left hand side action of $J$.

\begin{proposition}\label{openbranchsymplecto}
The map $\Phi_{\pi, \psi}: X(\phi) \to X(\psi, \pi)$ is a symplectomorphism on a dense, open
subset $X_I(\phi) \subset X(\phi).$
\end{proposition}

\begin{proof}
The map $\widetilde{\Phi}$ intertwines the Hamiltonian $\mathcal{L}$-actions on  $\big(T^*J \times E_{\mathcal{R}}(T^*K)\big)\q_{\! 0} J$ and $\big(T^*(J)^{\sc} \times E_{\mathcal{R}}(T^*K)\big)\q_{\! 0} J$.
Therefore,  it follows from the previous discussion that the following map is an isomorphism of Hamiltonian $J$-spaces: 
\begin{equation}
\Id \times \widetilde{\Phi}: E_{\mathcal{R}}(T^*L) \times \big((T^*J)_I \times E_{\mathcal{R}}(T^*K)\big)\q_{\! 0} J \to  E_{\mathcal{R}}(T^*L)\times  \big((T^*J)_I^{\sc} \times E_{\mathcal{R}}(T^*K)\big)\q_{\! 0} J.
\end{equation}
The reductions of these spaces by the diagonal action of $K$ are likewise isomorphic. 
\end{proof}

The subspace $X_I(\phi) \subset X(\phi)$ then inherits the Hamiltonian $\mathbb{T}_L \times \mathbb{T}_J \times \mathbb{T}_K$-action from  $X(\psi, \pi)$, and the momentum map $\mu_{\mathbb{T}_L \times \mathbb{T}_J \times \mathbb{T}_K}: X_I(\phi) \to \Delta(\pi, \psi)$ extends to a surjective, continuous map $\mu_{\mathbb{T}_L \times \mathbb{T}_J \times \mathbb{T}_K} \circ \Phi_{\pi, \psi}: X(\phi) \to \Delta(\pi, \psi).$  The following is a generalization to a chain of maps of compact Lie groups.

\begin{proposition}\label{chaincontract}
Let $$
\begin{CD}
K_0 @>\phi_1>> \ldots @> \phi_n>> K_n.
\end{CD}
$$ be a chain of maps of compact Lie groups.

\begin{enumerate}
\item  There is a surjective, continuous, proper map $$\Phi_{\vec{\phi}}: X(\phi_n \circ \ldots \circ \phi_1) \to X(\vec{\phi}) =  \big(X(\phi_n) \times \ldots \times X(\phi_1)\big)\q_{\! 0}\prod \mathbb{T}_i.$$
\item The map $\Phi_{\vec{\phi}}$ is  a symplectomorphism
on a dense, open subset $X_{\vec{I}}(\phi_n \circ \ldots \circ \phi_1) \subset X(\phi_n \circ \ldots \circ \phi_1)$.
\item There is a surjective, continuous map $\mu_{\prod \mathbb{T}_i} \circ \Phi_{\vec{\phi}}: X(\phi_n \circ \ldots \circ \phi_1) \to \Delta(\vec{\phi}),$
where $\Delta(\vec{\phi})$ is the fiber product cone $\Delta(\phi_1) \times_{\Delta_{1}} \ldots \times_{\Delta_{n-1}} \Delta(\phi_n),$ this is a momentum map for the action $\prod \mathbb{T}_i$ on $X_{\vec{I}}(\phi_n \circ \ldots \circ \phi_1).$
\end{enumerate}
\end{proposition}

\begin{proof}

For each $i$, the map $\Phi_i: T^*(K_i) \to T^*(K_i)^{\sc}$ is equivariant with respect to $K_{i-1}\times K_i$, and intertwines
the momentum maps for these spaces.  It follows that we can define a continuous, surjective map $$\Phi_{\vec{\phi}}: X(\phi_n \circ \ldots \circ \phi_1) \to \big(X(\phi_n) \times \ldots \times X(\phi_1)\big)\q_{\! 0}\prod \mathbb{T}_i$$ as in the proof of Proposition \ref{openbranchsymplecto}.  

We now build a dense, open subset $X_{\vec{I}}(\phi_n\circ \ldots \circ \phi_1) \subset X(\phi_n\circ \ldots \circ \phi_1)$
on which $\Phi_{\vec{\phi}}$ is a map of Hamiltonian $\mathbb{T}$-spaces.   We let $X_{I_{n-1}} \subset E_{\mathcal{R}}(T^*K_n)$
be the principal subspace for the action of $K_{n-1}$.  This is a smooth, dense, open Hamiltonian $K_{n-1} \times \mathbb{T}_{n-1}$
manifold, and the contraction map $$\Phi_{n-1}: E_{\mathcal{R}}(T^*K_n) \to {K_{n-1}}_{\ 0}{\!\ql}{} \big((TK_{n-1})_{I_{n-1}}^{\sc} \times E_{\mathcal{R}}(T^*K_n)]$$
restricts to an isomorphism of Hamiltonian $(K_{n-1} \times \mathbb{T}_{n-1})$-spaces on this subspace.   Now we let $X_{I_{n-1}, I_{n-2}} \subset X_{I_{n-1}}$ be the principal subspace for the action of $K_{n-1}$.  Continuing this way, we build $X_{\vec{I}} \subset E_{\mathcal{R}}(T^*K_n)$, and we let $X_{\vec{I}}(\phi_n \circ \ldots \circ \phi_1)$ be the implosion ${K_0}_{\ 0}{\!\ql}{} \big(E_{\mathcal{R}}((T^*K)_0)\times  X_{\vec{I}}\big)$. 

By construction, $\Phi_{\vec{\phi}}$ restricts to an isomorphism of Hamiltonian $\mathbb{T}$-spaces of $X_{\vec{I}}(\phi_n \circ \ldots \circ \phi_1)$ onto its image in $\big(X(\phi_n) \times \ldots \times X(\phi_1)\big)\q_{\! 0}\prod \mathbb{T}_i$. This is a symplectomorphism on the $K_0$ principal stratum.  
\end{proof}

\subsection{Contraction of projective branching varieties}\label{projcontractsection}

 The map $\Phi_{\pi, \psi}: X(\pi) \to X(\psi, \pi)$ is $\mathbb{T}_L \times \mathbb{T}_K$
equivariant, and intertwines the momentum maps of this action. It follows that there is a surjective, continuous map
\begin{multline}
\Phi_{\pi, \psi}: X_{\eta, \lambda}(\phi) \to \big(X(\pi) \times X(\phi)\big)\q_{(-\eta, 0, -\lambda)} (\mathbb{T}_L \times \mathbb{T}_J \times \mathbb{T}_K)\\ \cong L _{\ 0}{\!\qlb}{} \big(\mathcal{O}_{-\eta} \times (T^*J)^{\sc} \times \mathcal{O}_{-\lambda}\big)\qb_{\! 0} J.
\end{multline}

\begin{proposition}
The symplectic reduction $L _{\ 0}{\!\qlb}{} \big(\mathcal{O}_{-\eta} \times (T^*J)^{\sc} \times \mathcal{O}_{-\lambda}\big)\qb_{\! 0} J$ can be identified with $X_{\eta, \lambda}(\psi, \pi)$.  
\end{proposition}

\begin{proof}
We may view both spaces as a reduction of $E_{\mathcal{R}}(T^*L)\times (T^*J)^{\sc}\times E_{\mathcal{R}}(T^*K)$ by $\mathbb{T}_L \times L\times J\times \mathbb{T}_K$ at momentum level $(\eta, 0, 0, \lambda),$ so this is an application of the results of Sjamaar-Lerman \cite{SjL} and Kempf-Ness \cite{KN}. 
\end{proof}

We can perform the reduction of the previous proposition in stages, as a consequence the momentum image of residual $\mathbb{T}_J$ action on $X_{\eta, \lambda}(\psi, \pi)$ is the set of all triples of the form $(\eta, \beta, \lambda) \in \Delta(\psi)\times_{\Delta_J}\Delta(\pi)$, which is the fiber polytope of $\pi_L\times\pi_K$ over $(\eta, \lambda).$  

\begin{proposition}\label{projcontract}
The map $\Phi_{\pi, \psi}: X(\phi)\q_{-\lambda}\mathbb{T}_K \to X(\psi, \pi) \q_{-\lambda}\mathbb{T}_K$
is a symplectomorphism on a dense open subset.  If $\eta$ lies in the image of the $J$-principal stratum 
of  $\mathcal{O}_{-\lambda}$ under $d\pi^*$, then $$\Phi_{\pi, \psi}: X_{\eta, \lambda}(\phi) \to \big(X(\pi) \times X(\psi)\big)\qb_{\! (-\eta, 0, -\lambda)} (\mathbb{T}_L \times \mathbb{T}_J \times \mathbb{T}_K)$$ is a symplectomorphism on a dense, open subset.
\end{proposition}

\begin{proof}
Consider the symplectic horospherical contraction of $\mathcal{O}_{-\lambda}$ with respect to the action of $\psi: J \to K$:
\begin{equation}
\Phi_{ \eta}: \mathcal{O}_{-\lambda} \cong J \ql \big(T^*(J) \times \mathcal{O}_{-\lambda}\big) \to K \backslash \big(T^*(J)^{\sc} \times \mathcal{O}_{-\lambda}\big).
\end{equation}
This map is a symplectomorphism on the $J$-principal subspace of $\mathcal{O}_{-\lambda}.$  
\end{proof}

\subsection{The residual $\mathbb{T}_J$ action}\label{residual}

The contraction map $\Phi_{\psi, \pi}$ constructed on the affine branching variety $X(\phi)$ produces a symplectomorphism $$\Phi_{\psi, \pi}: X_I(\phi) \cong L _{\ 0}{\!\qlb}{}\big(E_{\mathcal{R}}(T^*L) \times (T^*J)_I \times E_{\mathcal{R}}(T^*K)\big)\qb_{\! 0} J.$$  The space on the right hand side carries a Hamiltonian action by $\mathbb{T}_J$, we derive a formula to compute this action on $X_I(\phi).$

The map $\Phi_{\psi, \pi}$ takes a  class $\big((p, w),(k, v)\big)$ to the equivalence class of $\big((p, w),(1, d\psi^*(k v)),(k, v)\big)$, where $d\psi^*: \mathfrak{k}^* \to \mathfrak{j}^*$ is induced by $\psi: J \to K.$  Following Section \ref{gen}, we see that an element $t \in \mathbb{T}_J$ acts on $\big((p, w),(1, d\psi^*(k v)),(k, v)\big)$ as follows: 
\begin{equation}
t \star \big((p, w),(1, d\psi^*(kv)),(k, v)\big) = \big((p, w),(h^{-1}t h,  d\psi^*(k v)),(k, v)\big).
\end{equation}
Here $h \in J$ is any element such that $h  d\psi^*(k v) \in \Delta_J.$  Pulling this back through the isomorphism $\Phi_{\psi, \pi}$, we obtain the following formula for that $\mathbb{T}_J$ action:
\begin{equation}
t \star \big((p, w),(k, v)\big) = \big((p, w),(\psi (h)^{-1}\psi(t)^{-1}\psi(h) k, v)\big).
\end{equation}

\begin{example}[Polygons, bending flows, and $\SU(2)$ diagonal branching]
We return to the case of $\SU(2)$ diagonal branching and the projective branching varieties $X_{\vec{r}}(\delta_n)$.  As a consequence  of Propositions \ref{chaincontract} and \ref{projcontract} there is a contraction map $\Phi_{\tree}: X_{\vec{r}}(\delta_n) \to X_{\vec{r}}(\tree)$ and an associated integrable system in $X_{\vec{r}}(\delta_n)$ associated to each trivalent tree $\tree$ with $n$ leaves.  As explained in \cite{HMM}, these integrable systems have a natural interpretation as operations on Euclidean polygons.  Each tree $\tree$ can be viewed as the dual complex to a triangulation of a model $n$-gon, where each edge $e \in E(\tree)$ is associated to a specific diagonal $d_e$. The action of $\U(1)$ corresponding to $e$ on a point $p$ in (a dense open subset of) $X_{\vec{r}}(\delta_n)$ can be interpreted as bending the polygon associated to $p$ along the diagonal $d_e$.  A similar integrable system is described on the Grassmannian variety in \cite{HMM}. 
\end{example}

\subsection{The branching gradient flow}

In both the projective and affine cases we have discussed the existence of an explicit branching contraction map $\Phi_{\vec{\phi}}$ associated to a chain of reductive group maps $\vec{\phi}$. Now we observe that this map is the flow at time $1$ of a gradient-Hamiltonian vector field on the flat family defined by the branching valuations  $\vec{h} \in C(\vec{\phi})$ (see Proposition \ref{chainbranch}).  

\begin{proposition}\label{branchflow}
Let $X$ be a projective or affine branching variety associated to a map $\phi: G_0 \to G_n$ of reductive groups equipped with a K\"ahler structure from an embedding in $\mathbb{P}^M$ or $\mathbb{A}^N$, respectively.   Let $\vec{\phi}$ be a chain of maps, let $\vec{h}$ be an interior point in $C(\vec{\phi})$, and finally let $E_{\vec{h}}$ be the associated $1$-parameter family.   The flow at time $1$ of the gradient Hamiltonian vector field on this family extends to give the branching contraction map $\Phi_{\vec{\phi}}.$
\end{proposition}

\begin{proof}
The branching variety $X$ is $G_0$ quotient of a product $G_0\q U_0\times G_n\q U_n$ or $\mathcal{O}_{\eta}\times \mathcal{O}_{\lambda}$ which we denote $\bar{X}.$  We trace the steps of Theorem \ref{lastref} to embed $\bar{X}$ in $\bar{X}\times \prod_{i =1}^n S_{G_i}$ and pass to the base change $\mathbb{A}^1\times_{\prod_{i =1}^n G_i} (\bar{X}\times \prod_{i =1}^n S_{G_i})$.  The flow of the gradient-Hamiltonian vector field on $\mathbb{A}^1\times_{\prod_{i =1}^n G_i} \prod_{i =1}^n S_{G_i}$ is invariant under the action on this family by $\prod_{i =1}^m K_i \times K_i$, where $K_i \subset G_i$ is the compact part, so this is also the case on the non-singular locus of $\mathbb{A}^1\times_{\prod_{i =1}^n G_i} (\bar{X}\times \prod_{i =1}^n S_{G_i}).$  As a consequence, this flow descends and gives a map which coincides with $\Phi_{\vec{\phi}}.$  
\end{proof}

\section{The Gel'fand-Tsetlin system}\label{GTsystem}

We reproduce the Gel'fand-Tsetlin system in a coadjoint orbit of $\U(n)$ by applying our branching construction.  What follows is an application of our symplectic contraction operation to the following chain of groups:

$$
\begin{CD}
\U(1) @>i_1>> \U(2) @>i_2>> \ldots @>i_{n-1}>> \U(n-1) @>i_{n-1}>> \U(n),\\
\end{CD}
$$

\noindent
where $i_k:\U(k) \to \U(k+1)$  is the upper left diagonal inclusion of $\U(k)$ in $\U(k+1)$ (this construction goes through without change for any chain of closed subgroups as in \cite[Section 5]{GSt}).  We choose the standard Weyl chamber $\Delta_k$ for each $\U(k)$, this is identified with the diagonal matrices in $\mathfrak{u}(k)^*$ with weakly decreasing entries down the diagonal.

A coadjoint orbit $\mathcal{O}_{\lambda} \subset \mathfrak{u}(n)^*$ can be considered as a projective branching space for the inclusion $1 \subset \U(n)$, it can be identified with the following symplectic reduction:

\begin{equation}
\big(T^*\U(1) \times \ldots \times T^*\U(n-1)\times \mathcal{O}_{\lambda}\big)\qb_{\! 0} \prod_{k = 1}^{n-1} \U(k).\\
\end{equation}

\noindent
Proposition \ref{chaincontract} then produces a contraction map:
\begin{equation}
\Phi: \mathcal{O}_{\lambda} \to \mathcal{O}_{\lambda}^{\sc} = \big((T^*\U(1))^{\sc} \times \ldots \times (T^*\U(n-1))^{\sc}\times \mathcal{O}_{\lambda}\big)\q_{\! 0}\prod_{k = 1}^{n-1} \U(k).
\end{equation}

The map $\Phi$ is a symplectomorphism on the subset $\mathcal{O}_{\lambda}^o \subset \mathcal{O}_{\lambda}$
of points $p$ such that the $\U(k)$-coadjoint orbit of $di_k^*\circ \ldots \circ di_{n-1}^*(p)$ intersects the principal face for the $\U(k)$ action on $\mathcal{O}_{\lambda}$ defined through $i_{n-1}\circ \ldots \circ i_k.$ 

By Proposition \ref{chaincontract} the space $\mathcal{O}_{\lambda}^{\sc}$ has a Hamiltonian action of $\mathbb{T} = \prod_{k = 1}^{n-1} \mathbb{T}_k.$  Any element $t \in \mathbb{T}_k \subset \mathbb{T}$ acts on $p \in \mathcal{O}_{\lambda}$ by the formula $t\star p = hth^{-1}p$, where $h \in \U(k)$ moves $di_k^*\circ \ldots \circ di_{n-1}^*(p)$ into the Weyl chamber $\Delta_k$.  For the chain of maps $i_1, \ldots, i_{n-1}$, this is successive diagonalization of each image $di_k^*\circ \ldots \circ di_{n-1}^*(p)$, and the condition that the resulting diagonal matrix be in the Weyl chamber requires that the entries are weakly decreasing down the diagonal (compare with \cite[Section 5]{GSt}).  The collection of these diagonalized elements defines the image of $p$ under the momentum map $\mu_{\mathbb{T}}: \mathcal{O}_{\lambda} \to \Delta_{\lambda}(\vec{i}) \subset \prod_{k =1}^{n-1} \Delta_k$.

\begin{definition}
A Gel'fand-Tsetlin pattern $\bold{x}$ of size $n$ is a triangular array of $\binom{n}{2}$ numbers $x_{i,j}$, $i \leq j$ such that $x_{i,j} \geq x_{i,j-1} \geq x_{i+1, j}$ for all $i, j$, see Figure \ref{GTpattern}.  If we fix a top row the set of all such patterns is called the Gel'fand-Tsetlin polytope (for that top row),
\end{definition}

\begin{figure}[htbp]
\begin{tikzpicture}
\draw [->] (-7,1.2) -- (-7.8, 0);
\draw [->] (-8, 0) -- (-8.8, 1.2);
\draw [->] (-6,2.8) -- (-6.8, 1.6);
\draw [->] (-7,1.6) -- (-7.8, 2.8);
\draw [->] (-8,2.8) -- (-8.8, 1.6);
\draw [->] (-9,1.6) -- (-9.8, 2.8);
\draw [->] (-5,4.4) -- (-5.8, 3.2);
\draw [->] (-6,3.2) -- (-6.8, 4.4);
\draw [->] (-7,4.4) -- (-7.8, 3.2);
\draw [->] (-8,3.2) -- (-8.8, 4.4);
\draw [->] (-9,4.4) -- (-9.8, 3.2);
\draw [->] (-10,3.2) -- (-10.8, 4.4);
\node at (-7.9, -.2) {$x_{1, 1}$};
\node at (-6.9, 1.4) {$x_{2, 2}$};
\node at (-8.9, 1.4) {$x_{1, 2}$};
\node at (-5.9, 3) {$x_{3, 3}$};
\node at (-7.9, 3) {$x_{2, 3}$};
\node at (-9.9, 3) {$x_{1, 3}$};
\node at (-4.9, 4.6) {$x_{4, 4}$};
\node at (-6.9, 4.6) {$x_{3, 4}$};
\node at (-8.9, 4.6) {$x_{2, 4}$};
\node at (-10.9, 4.6) {$x_{1,4}$};
\end{tikzpicture}
\caption{A Gel'fand-Tsetlin pattern of size $4$.}\label{GTpattern}
\end{figure}
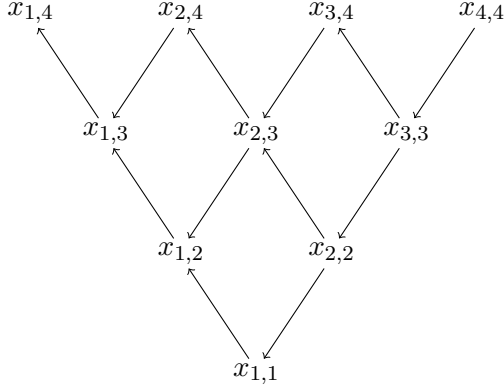

\begin{proposition}
The image $\Delta_{\lambda}(\vec{i})$ is the Gel'fand-Tsetlin polytope with top row $\lambda$.  If $\lambda$ is integral, $\mathcal{O}_{\lambda}^{\sc}$ is the projective toric variety associated to $\Delta_{\lambda}(\vec{i}).$  
\end{proposition}

\begin{proof}
We will perform the contraction by the chain $\vec{i}$ on the imploded cotangent bundle $T^*\U(n)$, which we identify with $\GL_n(\C) \q U.$ Propositions
\ref{chaincontract}, \ref{chainbranch} and \ref{branchflow} imply that there is a degeneration and a contraction of $\GL_n(\C) \q U,$ which we can consider to be $X(i)$ for $i: \C^* \to \GL_n(\C)$, to the variety $X(\vec{i})$. The latter is an affine variety with the following multigraded coordinate algebra:

\begin{equation}
\C[X(\vec{i})] = \bigoplus_{\lambda_i \in \mathfrak{X}^+_{GL_i(\C)}} \Hom(M_{\lambda_1}, M_{\lambda_2}) \otimes \ldots \otimes \Hom(M_{\lambda_{n-1}}, M_{\lambda_n}).\\
\end{equation}

\noindent
Here $\lambda_i$ is an $i$-tuple of weakly decreasing integers, in particular $\lambda_1$ is a single number representing a character of $\C^*$. As a consequence of the Pieri rule, each multigraded summand of this algebra has dimension $1$ or $0$.  As it must be finitely generated, it follows that $\C[X(\vec{i})]$ is the affine semigroup algebra associated to the monoid $P(\vec{i})$.  This branching rule also implies that a tuple $\vec{\lambda}$ is in $P(\vec{i})$ if and only if its entries interlace, it follows that $\Delta(\vec{i})$ is the cone of Gel'fand-Tsetlin patterns of size $n$.  Since this cone is normal, $X(\vec{i})$ is an affine toric variety, the residual $\mathbb{T}\times \mathbb{T}_n$ action on $X(\vec{i})$ has momentum image equal to $\Delta(\vec{i})$, and the contraction map $\Phi: \GL_n(\C) \q U \to X(\vec{i})$ produces a dense open integrable system in $\GL_n(\C) \q U$ with momentum image a dense open subset of $\Delta(\vec{i}).$  Now it follows that $\Delta_{\lambda}(\vec{i})$ is the polytope of Gel'fand-Tsetlin patterns with top row equal to $\lambda$, and $\mathcal{O}_{\lambda}^{\sc}$ is the associated projective toric variety when $\lambda$ is integral. 
\end{proof}

By Proposition \ref{chainbranch}, an integral choice $\vec{h}$ from the interior of the cone of valuations $C(\vec{i})$ defines a one parameter flat degeneration $E_{\vec{h}}(\lambda)$ of $\mathcal{O}_{\lambda}$ to $\mathcal{O}_{\lambda}^{\sc}$. Finally, by Proposition \ref{branchflow} the flow at time $1$ of the gradient-Hamiltonian vector field on this family extends to give the map $\Phi.$

\appendix 
\section{K\"ahler structures on $E_{\mathcal{L}}(T^*K)$ }  \label{appendix}
In \cite[\S 6]{GJS} it was shown that, for $K$ semi-simple and simply connected, $E_{\mathcal{R}}(T^*K)$ is isomorphic (as a Hamiltonian $K$-space) to $G\q U = \Spec(G)^U$, also known as the basic affine space, or the affinization of $G/U$.  We show here that this result is in fact true for arbitrary compact connected Lie groups $K$, by essentially the same proof as in \cite{GJS} -- this appendix should therefore just be seen as a minor comment.  Of course the corresponding statement also holds for $E_{\mathcal{L}}(T^*K)$.

In order to make this comparison, one needs to endow $G\q U$ with a K\"ahler structure.  This is done by embedding it in a Hermitian vector space.  Choose a finite set $\Pi$ of generators of $\mathfrak{X}_G^+$ (in \cite{GJS} these are required to be minimal, but we will not demand this), and look at $$E=\bigoplus_{\varpi\in \Pi} M_{\varpi}.$$  By using the Borel-Weil theorem, one can think of $E^*$ as a subspace of $\mathbb{C}[G]^U$, which generates the latter as a ring, and hence one has a $G$-equivariant embedding $G\q U\hookrightarrow E$.  If one chooses  highest-weight vectors $v_{\varpi}\in M_{\varpi}$ for all elements in $\varpi\in\Pi$, this embedding is determined by sending the identity (mod $U$) to $\sum v_{\varpi}$.  One should now look at $X=\overline{T\sum v_{\varpi}}\subset E$ where $T$ acts through the weights $-\varpi$.  This is an affine toric variety (in fact contained in the invariant subspace $E^U$).

Moreover, one can equip $E$ with the unique Hermitian inner product that is $K$-invariant, satisfies $|\!| v_{\varpi} |\! |=1$, and is real-valued on all real linear combinations of the $v_{\varpi}$.  This gives $G\q U$ the structure of Hamiltonian $K$-space, and makes $X$ into a (likely singular) symplectic toric space for the compact torus $\mathbb{T}$.  In particular, it will have a momentum map onto $-\Delta^{\vee}$, and moreover, very important for us, a section of this momentum map $-\Delta^{\vee}\rightarrow X$.  We shall denote minus composed with this section as $s$ -- the minus sign comes from the use of the $\mathcal{R}$-action, we can think of this $s$ as a section of the momentum map for the action of the torus by inverses.

If $K$ is semi-simple and simply connected, and moreover the $\varpi$ are the fundamental weights (the only minimal choice  of generators for $\mathfrak{X}_G^+$), then $X=E^U$ (which is of course an affine space), and one can write down an explicit formula for this section $s$ -- this is given in \cite[Formula (6.6)]{GJS}.  But indeed such a section always exists (though writing down an explicit formula for it will be harder in general).  It is perhaps more familiar to algebraic geometers, who know its image as the \emph{non-negative part}, $X_{\geq 0}$, of $X$ (see e.g. \cite[Theorem 12.2.5 and Exercise 12.2.8]{Coxandco}).  Of relevance for us is that it is a homeomorphism onto its image, which is a diffeomorphism restricted to each face (whose target is considered as smooth in its stratum).  Moreover, with the choices we have made, its image will lie in the real linear span of the $v_{\varpi}$.

It is now a matter of observing that with this section $s$, even without a formula for it, the reasoning of \cite[\S 6]{GJS} still carries through.  

First, note that $s$ extends uniquely to a $(K\times \mathbb{T})$-equivariant map $\mathcal{F}: K\times \Delta^{\vee}\rightarrow G\q U \subset E$.

\begin{theorem}[{\cite[Proposition 6.8]{GJS}}]
Let $K$ be a compact connected Lie group.  We have \begin{enumerate}
\item \label{een} $\mathcal{F}$ induces a closed embedding $f: E_{\mathcal{R}}(T^*K)\hookrightarrow E$.
\item \label{twee} On each stratum this map is a smooth symplectomorphism.
\item \label{drie}The image is $G\q U \subset E$.
\end{enumerate}
\begin{proof}
The proof of (\ref{een}) and (\ref{drie}) carry over verbatim from \cite{GJS}.  For (\ref{twee}), the smoothness of $f$ on all strata still follows from the same property of the section $s$.  To show that it is a symplectomorphism, we still follow the reasoning of \cite{GJS}, but without invoking an explicit formula for $s$.  Indeed, the symplectic structure $\omega_E$ on $E$ can be written as $\omega_E=d\beta_E$, where $(\beta_E)_v(w)=-\frac{1}{2}\Im\langle v,w\rangle$.  On the other hand, the symplectic structure on a stratum of $E_{\mathcal{R }}(T^*K)$ corresponding to a face $\sigma$ at a point $(\overline{1},\lambda)$ is given by 
$d\beta_{\sigma}$ (\cite[Lemma 4.6]{GJS}), where $(\beta_{\sigma})_{(\overline{1},\lambda)}(\overline{\xi},\mu)=\lambda(\xi)$.  Note that by equivariance we only need to compare $\beta_{\sigma}$ and $f^*(\beta_E)$ at such $(\overline{1},\lambda)$.  

We now have $$(f^*\beta_E)_{(\overline{1},\mu)}(\overline{\xi},\mu)=(\beta_E)_{s(\lambda)}(\mathcal{F}_*(\xi,\mu)),$$ and $$\mathcal{F}_*(\xi,\mu)=\frac{d}{dt}\Big|_{t=0}\Big(\exp(t\xi)s(\lambda+t\mu)\Big)=\xi_E(s(\lambda))+ \frac{d}{dt}\Big|_{t=0}s(\lambda+t\mu).$$

We therefore have 
$$\begin{aligned}
(f^*\beta_E)_{(\overline{1},\lambda)}(\overline{\xi},\mu)&= \frac{1}{2} \Im \Big\langle \xi_E(s(\lambda))+ \frac{d}{dt}\Big|_{t=0}s(\lambda+t\mu), s(\lambda)\Big\rangle \\
&= \frac{1}{2}\Im\Big\langle \xi_E(s(\lambda)),s(\lambda)\Big\rangle + \frac{1}{2}\frac{d}{dt}\Big|_{t=0}\underbrace{\Im\Big\langle s(\lambda+t\mu),s(\lambda)\Big\rangle}_{=0\text{ as } s(.)\in X_{\geq0}}\\ &= \mu(s(\lambda))(\xi)\\
&=\lambda(\xi)
\\
&=(\beta_{\sigma})_{(\overline{1},\lambda)}(\overline{\xi},\mu),
\end{aligned}$$
and hence $f^*\omega_E$ coincides with the symplectic structure on $E_{\mathcal{R}}(T^*K)$.
This proves (\ref{twee}).
\end{proof}

\end{theorem}

\def\cprime{$'$}

\end{document}